\documentclass{article}

\usepackage[xetex]{geometry}
\RequirePackage[hidelinks]{hyperref}

\title{Regularity of fractional Schrödinger equations and sub-Laplacian Fourier multipliers on the Heisenberg group}
\author{Aksel Bergfeldt}
\date{\today}

\RequirePackage{polyglossia}
\setmainlanguage{british}

\usepackage{biblatex}
\renewbibmacro{in:}{}
\DeclareFieldFormat{labelnumberwidth}{#1.}

\usepackage[math-style=ISO]{unicode-math}
\usepackage{fontspec}
\def\usemathfont{\fontspec{latinmodern-math.otf}}

\renewenvironment{abstract}{\small\noindent\textbf{Abstract:}\ }{\par\vspace{0em}}

\usepackage{sectsty}
\sectionfont{\mdseries}
\subsectionfont{\mdseries}
\subsubsectionfont{\mdseries\itshape}
\paragraphfont{\mdseries\itshape}

\removenolimits\int 
\usepackage{newunicodechar}
\newunicodechar{≤}{\leqslant}
\newunicodechar{≥}{\geqslant}
\def\capfigures{}

\def\dd{\symup d}
\def\de{\,\dd}
\def\eg{e.\,g.\@ }

\def\jap#1{\langle#1\rangle}
\def\BMO{\mathup{BMO}}
\def\tr{\mathup{tr}\mathop{}}
\def\ud#1{#1\hspace{-.1em}\text{\usemathfont\symbol{"0323}}\hspace{.1em}}
\def\argd{\hspace{.1em}{·}\hspace{.1em}}
\def\Schwartz{\mathscr S}
\def\Continuous{\mathscr C}
\def\Smooth{\Continuous^∞}

\def\supp{\mathop{\mathup{supp}}}

\def\combin#1#2{\bigg(\begin{array}{@{}c@{}}#1\\[-.2em]#2\end{array}\bigg)}
\def\sfrac#1#2{{\textstyle\frac#1#2}}
\def\floor#1{\lfloor#1\rfloor}
\def\ceil#1{\lceil#1\rceil}
\def\pref#1{(\ref{#1})}

\def\go{\boxplus}
\def\inpr#1{\left\langle#1\right\rangle}
\def\Rinpr#1{\inpr{#1}_{L^2(ℝ^d)}}
\def\Hinpr#1{\inpr{#1}_{L^2(ℍ^d)}}
\def\Hnorm#1{\left\|#1\right\|_{L^2(ℍ^d)}}
\def\FOERSTA{\hspace{-1.5em}}
\def\ANDRA{\phantom{=}\,\,}
\def\TERM{\hspace{1.5em}}
\def\db#1{\frac{\dd#1}#1}

\def\F{\mathscr F}
\def\G{\mathscr G}

\allowdisplaybreaks
\counterwithin{equation}{section}

\RequirePackage{amsthm}
\newtheoremstyle{uppraett}
  {\topsep}   % Space above
  {\topsep}   % Space below
  {\upshape}  % Body font
  {}       % Indent amount (empty value is the same as 0pt)
  {\bfseries} % Theorem head font
  {}          % Punctuation after theorem head
  {5pt plus 1pt minus 1pt} % Space after theorem head
  {\thmname{#1}\thmnumber{ #2}.\thmnote{ (#3)}} % Theorem head spec

\theoremstyle{uppraett} % För Arxiv
\newtheorem{theorem}{Theorem}[section]
\newtheorem{lemma}[theorem]{Lemma}
\newtheorem{corollary}[theorem]{Corollary}
\newtheorem{definition}[theorem]{Definition}

\begin{document}
\def\Re{\mathrm{Re}\,}
\def\Im{\mathrm{Im}\,}

\thispagestyle{empty}
\vspace*{-5mm}
\noindent
\begin{flushleft}
  \makeatletter%
  \fontsize{20}{24}\selectfont\@title%
  \makeatother%
\end{flushleft}
\vspace{.5em}\par%
\noindent\@author\ 2026-04-09%
\vspace{2.5em}\par%

\begin{abstract}
  We define functions of the sub-Laplacian $∆$ on the Heisenberg group $ℍ^d$ as Fourier multipliers. In this setting, we show that the solution $u$ of the free fractional Schrödinger equation $i∂_tu + (-∆)^νu = 0, u|_{t=0} = u_0$, for any $ν > 0$, satisfies the Hardy space estimate that 
  \[
    \|u(t,\argd)\|_{H^p(ℍ^d)} ≤ C_p (1 + t)^{Q|1/p-1/2|}\|(1-∆)^{νQ|1/p-1/2|}u_0\|_{H^p(ℍ^d)},
  \]
  with $Q = 2d + 2$, for all $p ∈ (0,∞)$, and the corresponding estimate with $p=∞$ in $\BMO(ℍ^d)$. This is done via a general regularity result for parameter dependent sub-Laplacian Fourier multipliers. We prove also that Bessel potential spaces on the Heisenberg group correspond to Sobolev spaces in the same way as in Euclidean space, also for Hardy spaces.
\end{abstract}


\medskip

\section{Introduction}
Let $ℍ^d$ denote the Heisenberg group of real dimension $2d + 1$. We shall use standard variable names  $v = (z,s) = (x, y, s) ∈ ℍ^d$ where $z = (x,y) ∈ ℝ^d × ℝ^d$ and $s∈ℝ$. The  group operation is 
\[
  (z,s)\go(z'\!,s')  = \left(z + z'\!,\ s + s' + \textstyle\frac12(y·x' - x·y')\right),
\]
and we denote the inverse $(z,s)^{-1} = (-z,-s)$ and the neutral element $0$.

The left-invariant vector fields are found by derivatives with right translations; $\lim_{t→0}(f(w\go tv)-f(w))/t$,
for some fixed $v∈ℍ^d$. With $tv$ we mean here the Euclidean scaling of $v$ with $t$ in $ℝ^{2d+1}$. 
If we take $v$ to be the coordinate units, this gives us the left-invariant vector fields $X = ∇_x + \frac12y∂_s$, $Y = ∇_y - \frac12x∂_s$ and $S = ∂_s$. 

With these vector fields we can construct the Heisenberg group sub-Laplacian $∆ = ∑_{i=1}^d ( X_i^2+Y_i^2)$. 
The present work studies the free fractional Schrödinger equations created with this sub-Laplacian;
\begin{equation*}
  \left\{\ 
    \begin{aligned}
      i∂_t u + (-∆)^νu &= 0\\
      u|_{t=0} &= u_0,
    \end{aligned} 
  \right.
\end{equation*}
for $ν > 0$ and where $u : [0,∞) × ℍ^d → ℂ$ is unknown and $u_0 : ℍ^d → ℂ$ is the initial data. The purpose of the study is to determine the change in regularity when going from initial data to solution of this problem. In other words, to count the number of bounded derivatives that are lost, in Lebesgue and Hardy spaces. To make sense of the equation for non-integer $ν$, we define functions of $∆$ using the Fourier transform on $ℍ^d$, see Section~\ref{Fourieranalysavdelning}.

The solution operator of the Schrödinger system above is given by a $t$-dependent sub-Laplacian Fourier multiplier, and hence this regularity statement is a special case of a more general question about the regularity of such operators. In this work we shall study the regularity properties of general multipliers $φ_t(-∆)$, for which the corresponding result for Schrödinger equations follows.

We shall take initial data in the Hardy space $H^p(ℍ^d)$ for $0 < p < ∞$  as well as in the space $\BMO(ℍ^d)$ of functions of bounded mean oscillation. For this range of function spaces we shall also show that Sobolev spaces created using the left-invariant vector fields on the Heisenberg group correspond to the description using Bessel potentials in the same way as in Euclidean space.

\subsection{Background and related work}
It is well known that the mapping of initial state to solution at a fixed time of the Schrödinger equation in Euclidean space is only $L^p$ bounded for $p = 2$, see L.~Hörmander in 1960 \cite{Hoermander}. General $L^p$ boundedness can however be achieved at the cost of regularity, and the full characterisation of this behaviour was subsequently given by A.~Miyachi in 1980 \cite{Miyachi}: The solution of the free fractional Schrödinger equation on $ℝ^d$ satisfies 
\[
  \|u(t,·)\|_{H^p(ℝ^d)} ≤ C_p(1 + t)^{d|1/p-1/2|}\big\|(1 - ∆_{ℝ^d})^{r/2}u_0\big\|_{H^p(ℝ^d)}\quad\text{if }r≥2νd\left|\textstyle\frac1p - \frac12\right|,
\]
for $0 < p < ∞$. For $p = ∞$, the corresponding condition holds instead for the space of functions of bounded mean oscillation. A fundamental property of the Bessel potential $(1-∆_{ℝ^d})^{r/2}$ here is that when $r$ is an integer, $\|(1 - ∆_{ℝ^d})^{r/2}u_0\|_{H^p(ℝ^d)}$ is estimated from above and below by $∑_{|α|≤r}\|∂^αu_0\|_{H^p(ℝ^d)}$, and the estimate above hence counts the number of lost derivatives when comparing the solution at time $t$ to the initial state.

\paragraph{Schrödinger equations}
Regarding the generalisation of the Schrödinger equation estimate to groups other than Euclidean space, we have recently seen great activity towards showing boundedness of the operator
\[
  T_t = (1 + t)^{-Q|1/p-1/2|}e^{it(-∆)^ν}(1 - ∆)^{-νQ|1/p-1/2|}
\]
in settings such as homogeneous groups, or even homogeneous spaces, see \cites{BuiDN, ChenDLY-Lp, BuiL, ChenDLY-Hp, BuiDD}. In the non-fractional case $ν = 1$, the boundedness of this operator on Lebesgue and Hardy spaces has been established in this general setting. The constant $Q$ above is the homogeneous dimension of the space, and $∆$ is a chosen sub-elliptic operator. Together with invertibility of the Bessel potential $(1 - ∆)^{νQ|1/p-1/2|}$, boundedness of this operator implies the generalisation of Miyachi's result. For fractional Schrödinger flows on the Heisenberg group, the strongest results are from three recent papers: R.~Bramati, P.~Ciatti, J.~Green and J.~Wright \cite{BramatiCGW}, showed that the fixed time operator
\[
   e^{i(-∆)^ν}(-∆)^{-r/2}χ^\mathup H(-∆),
\]
where $χ^\mathup H(-∆)$ is a high frequency cut-off, is bounded on the Lebesgue spaces $L^p(ℍ^d)$ for $1 < p < ∞$ if and only if $r ≥ Qν|1/p-1/2|$. The same year, T.~A.~Bui, Q.~Hong and G.~Hu \cite {BuiHH} showed that the operator
\[
  (1 + t)^{-Q|1/p-1/2|}e^{it(-∆)^ν}(-∆)^{-Q|1/p - 1/2|}χ^\mathup H(-∆)
\]
is bounded on $H^p(ℍ^d)$ for all $p ∈ (0,∞)$ and all $t$, and that the full operator $T_t$ above is bounded on $H^p(ℍ^d)$ for all $p ∈ (0,1)$ and all $t$. The year after, Bui, P.~D'Ancona and X.~T.~Duong \cite{BuiDD} showed the boundedness of $T_t$ on the Hardy spaces $H^p(ℍ^d)$ with $p ∈ (0,1]$ in the cases when $ν ∈ (0,1/2]$, as well as on $L^p(ℍ^d)$ with $p ∈ (1,∞)$ for all $ν$.

In the present work, we shall show the boundedness on $H^p(ℍ^d)$ for all $p ∈ (0,∞)$ and all $t$ of a class of operators that include all of the above. This includes in particular the boundedness on $H^1(ℍ^d)$ for operators with $ν > 1/2$, which has been missing in the literature so far. We shall also show the invertibility of Bessel potentials that implies the boundedness of the Schrödinger equation solution $u(t,·)$ such as in Miyachi's estimate. 

The methods of the cited accomplishments above rely primarily on the behaviour of the heat kernel associated to the sub-Laplacian, and are quite different from the approach of the present work. Here we shall instead build the analysis entirely upon the Fourier expansion of functions and operators on $ℍ^d$. This viewpoint has been elaborated by H.~Bahouri, J.-Y.~Chemin and R.~Danchin \cite{BahouriCD-Tempered, BahouriCD-Frequency}, where they put emphasis on how the frequency space decomposes into eigenspaces of harmonic oscillator Hamiltonians. The same techniques were also recently used by Bahouri and I.~Gallagher \cite{BahouriG-Local_dispersive} to form the Schrödinger flow, derive an explicit form of the Schrödinger kernel and study the dispersive properties of the Schrödinger equation on the Heisenberg group.

For the methods in this work, we shall be concerned mainly with the diagonal form of this expansion, which corresponds to a form of radiality of functions or distributions, and fully describes operators built as functions of the sub-Laplacian.

\paragraph{Sobolev space description}
For the Lebesgue spaces with $1 < p < ∞$, it was shown by G.~B.~Folland in 1975 \cite{Folland-Subelliptic} that the Sobolev space norm $∑_{|I|≤n}\|D^I\argd\|_{L^p(ℍ^d)}$ is for integer $|I|$ equivalent to the Bessel potential norm $\|(1-∆)^{n/2}\argd\|_{L^p(ℍ^d)}$. Here $D^I$ is a concatenation of left-invariant vector fields as in Theorem \ref{Besselpotentials} below. For Hardy spaces with $0 < p ≤ 1$ and for BMO, this fundamental correspondence does however so far appear to be missing in the literature. We shall derive this property as a consequence of a general multiplier theorem.

\paragraph{The half wave equation}
For the case $ν = 1/2$, the fractional Schrödinger equation becomes the half wave equation, and in this particular case it is known that the range of allowed $r$ of Miyachi above is not optimal. This is also the case for the corresponding equation on the Heisenberg group, as shown by D.~Müller and E.~M.~Stein \cite{MuellerStein}, with sharp results by Müller and A.~Seeger \cite{MuellerSeeger}.

\subsection{Main results}

In this study, we shall show the equivalent of Miyachi's estimate on the Heisenberg group. The main result is the following theorem, which mimics the Euclidean case, with the important difference that the dimension of space is replaced with the homogeneous dimension $Q = 2d + 2$.

\begin{theorem}\label{pd-s}
  For all $p ∈ (0,∞)$, the solution $u$ of the free fractional Schrödinger equation on $ℍ^d$, as above, satisfies
  \[
    \|u(t,\argd)\|_{H^p(ℍ^d)} ≤ C_p (1 + t)^{Q|1/p\, -\, 1/2|} \left\|(1 - ∆)^{νQ|1/p - 1/2|}u_0\right\|_{H^p(ℍ^d)},
  \]
  while for the $p = ∞$ case we have
  \[
    \|u(t,·)\|_{\BMO(ℍ^d)} ≤ C(1 + t)^{Q/2} \|(1-∆)^{νQ/2}u_0\|_{\BMO(ℍ^d)}.
  \]
\end{theorem}
We give the definition of the Hardy and BMO spaces on $ℍ^d$ in Section \ref{Hardyrumsavdelning}. 
Just as in the Euclidean case, the Bessel potentials $(1 - ∆)^{r/2}$ are formed by a form of Fourier multiplier, and our next theorem shows that these count the number of bounded derivatives in the expected way.

\begin{theorem}\label{Besselpotentials}
  For all $p ∈ (0,∞)$ and $n ∈ ℕ$ there are constants $c_p,C_p > 0$ such that for all $f ∈ H^p(ℍ^d)$,
  \begin{align*}
    c_p\|(1 - ∆)^{n/2}f\|_{H^p(ℍ^d)} &≤ ∑_{|I| ≤ n} \|D^If\|_{H^p(ℍ^d)} ≤ C_p\|(1 - ∆)^{n/2}f\|_{H^p(ℍ^d)},
  \end{align*}
  and likewise for $\BMO(ℍ^d)$.
\end{theorem}
Here the multi-index $I$ is a list $(i_1, \ldots, i_k)$ for some $k$, with $i_j ∈ \{1, \ldots, 2d + 1\}$ and $|I| = ∑_{j=1}^k μ(i_j)$ with $μ(i) = 2$ if $i = 2d + 1$ and otherwise $μ(i) = 1$. To $I$ corresponds $D^I = D_{i_1}\cdots D_{i_k}$ where $D_i = X_i$ and $D_{d+i} = Y_i$ for $i = 1,\ldots,d$, and $D_{2d + 1} = S$. For example, when $n = 2$ this shows that the norms $\|f\|_{H^p(ℍ^d)}$, $\|X_if\|_{H^p(ℍ^d)}$, $\|Y_if\|_{H^p(ℍ^d)}$, $\|X_iX_jf\|_{H^p(ℍ^d)}$, $\|X_iY_jf\|_{H^p(ℍ^d)}$, $\|Y_iX_jf\|_{H^p(ℍ^d)}$, $\|Y_iY_jf\|_{H^p(ℍ^d)}$ and $\|Sf\|_{H^p(ℍ^d)}$, for all $1 ≤ i,j ≤ d$, are all estimated from above by $\|(1 - ∆)f\|_{H^p(ℍ^d)}$.

These two theorems follow from a more general theory about Fourier multipliers created as functions of the sub-Laplacian that we shall develop. The main result about these objects, which we specify in Definition \ref{def multiplikator} below, is the following theorem.

\begin{theorem}\label{allmaenna s}
  Let $φ_t ∈ \Continuous^∞([0,∞))$ with $t ∈ [0,∞)$ satisfy for some $p ∈ (0,∞)$, some $γ ≥ 0$ and some multi-index $K$ that
  \[
    |∂^mφ_t(σ)| ≤ C_m\,(1 + t)^m(1 + σ)^{-γQ|1/p - 1/2| - |K|/2}σ^{m(γ - 1)}
  \]
  for all $σ ∈ [0,∞)$ and all $m ∈ ℕ$. Then for any multi-indices $I,I'$ such that $K$ is the concatenation of $I$ and $I'$, it holds that for all $f ∈ H^p(ℍ^d)$,
  \[
    \|D^Iφ_t(-∆)D^{I'}f\|_{H^p(ℍ^d)} ≤ C_p(1 + t)^{Q|1/p - 1/2|}\|f\|_{H^p(ℍ^d)}.
  \]
  If the requirement above holds with $p = ∞$, we instead have for all $f∈\BMO(ℍ^d)$ that
  \[
    \|D^Iφ_t(-∆)D^{I'}f\|_{\BMO(ℍ^d)} ≤ C(1 + t)^{Q/2}\|f\|_{\BMO(ℍ^d)}.
  \]
\end{theorem}

Examples of such operators include
\begin{itemize}
\item
  A kind of sub-Laplacian Mihlin multipliers are the $φ(-∆)$ for which $|∂^mφ(σ)| ≤ C_mσ^{-m}$. For these operators we may take $γ = t = |K| = 0$, and thus deduce their boundedness on $H^p(ℍ^d)$ for any $p$ as well as on $\BMO(ℍ^d)$ from this theorem. This result is well known; see e.\,g.\@ \cite{MuellerS} where it is shown using the spectral decomposition of $∆$.
\item
  The solution to $i∂_tu + (-∆)^nu = 0$, $u(0,\argd) = u_0$ for a positive integer $n$ is $u(t,\argd) = e^{it(-∆)^n}u_0$. The regularity statement
  \[
    \|u(t,\argd)\|_{H^p(ℍ^d)} ≤ C_p (1 + t)^{Q|1/p-1/2|}\|(1-∆)^{r/2}u_0\|_{H^p(ℍ^d)}
  \]
  with $r ≥ nQ|1/p-1/2|$, and the corresponding bound for $\BMO$, then follow from this theorem if one applies it to the operator with $φ_t(σ) = e^{itσ^n}(1 + σ)^{-r/2}$, taking $γ = n$ and $|K| = 0$, and applies Theorem \ref{maottlig multiplikator-s} below to the Bessel potentials.
\item
  The operators $D^I(1 - ∆)^{-n/2}$ with $|I| ≤ n$ satisfy the condition with $γ = t = |I'| = 0$ and any $p$. Together with Theorem \ref{maottlig multiplikator-s}, this shows that $\|D^If\|_{H^p(ℍ^d)} ≤ C_p\|(1-∆)^{n/2}f\|_{H^p(ℍ^d)}$ holds for all Schwartz $f$, and likewise for BMO. This thus provides the second estimate of Theorem \ref{Besselpotentials}.   
\end{itemize}
To deduce Theorem \ref{pd-s} for non-integer $γ$ and the first estimate of Theorem \ref{Besselpotentials}, a little more work is required. This is done in Section \ref{tillaempningsavdelning}.

\paragraph*{Outline}
In Section \ref{Fourieranalysavdelning} we state the basic Fourier analysis on the Heisenberg group, define sub-Laplacian Fourier multipliers and derive a number of $L^2$ estimates and properties of these operators. We also prove a Taylor's theorem on $ℍ^d$. In Section \ref{nyckeluppskattningsavdelning}, we state the refinement of the frequency decomposition of $φ_t(-∆)$ into a diagonal Laguerre polynomial expansion, and prove the key Lemma \ref{nyckelhs}, which relates position multipliers with frequency derivatives. As an application, we show that sub-Laplacians of moderate growth are closed on the Schwartz space on $ℍ^d$. In Section \ref{Hardyrumsavdelning}, we state the definition of Hardy spaces on $ℍ^d$, and provide a related commutation lemma. Theorem \ref{allmaenna s} is proved in Section \ref{allmaenna ss avdelning}. In Section \ref{tillaempningsavdelning}, we prove Theorem \ref{pd-s} and Theorem \ref{Besselpotentials} as corollaries of Theorem \ref{allmaenna s}.

The proof of Theorem \ref{allmaenna s} follows the same general steps as Miyachi's proof. It uses the atomic as well as square function characterisation of Hardy spaces, which was developed for the Heisenberg group and other homogeneous groups by Folland and Stein \cite{FollandS}. A major new contribution is the mentioned Lemma \ref{nyckelhs}. The counterpart for Euclidean space of this lemma is trivial, since the Euclidean Fourier transform maps $\Schwartz(ℝ^d)$ to itself, but the nature of position multipliers on general homogeneous groups appears to be highly dependent on the specific group structure. A new general method for showing $H^1$ boundedness has also been developed, which was needed since the necessary quasi-Banach to Banach interpolation tools for Triebel–Lizorkin spaces on the Heisenberg group are still unavailable, to the author's knowledge.

\paragraph*{Notation and conventions} It will be convenient to denote implicit variables with a dot under the letter. We may \eg write $\|f(v,\ud w)\| = \|f(v,\argd)\|$ meaning that the norm is with respect to the second argument of $f$. We shall frequently denote positive constants with the letter $C$, which may be different in each instance. We will also write $A \lesssim B$ if there is a $C > 0$ such that $|A| ≤ CB$, and $A \simeq B$ means that $A \lesssim B$ and $B \lesssim A$. Inner products of $L^2$-functions $f,g$ are denoted by $\langle f, g \rangle_{L^2} = ∫\overbar fg$. We shall write the action of a distribution $d$ on a test function $f$ as $(d \mathbin| f) = (d \mathbin|_v f(v))$.

\section{Fourier analysis and \textit{L\textsuperscript2} estimates}\label{Fourieranalysavdelning}
We start by recalling some basic properties of the Heisenberg group and the harmonic analysis of its Schrödinger representation. We refer to \cite{Taylor} and \cite{FollandS} for the statements given without proof.

The Heisenberg group operation commutes with the dilation operation $δ_t : (z,s) ↦ (tz, t^2s)$ for nonzero $t ∈ ℝ$. This property makes $ℍ^d$ into a homogeneous group, with the norm $|v| = (|z|^4 + s^2)^{1/4}$. A function $p$ such that $p\circ δ_t = t^np$ is said to be of homogeneous degree $n$. A (differential) operator $A$ such that $A[f\circ δ_t] = t^m Af \circ δ_t$ is said to be of homogeneous degree $m$. For such $p$ and $A$, $Ap$ is homogeneous of degree $n - m$. If $p$ is a polynomial, $Ap = 0$ if $m > n$.

We shall include all the left-invariant vector fields of $ℍ^d$ in the combined vector $D = (D_1,\ldots,D_{2d+1})$, and only the so-called horizontal vector fields of $X$ and $Y$ in the vector $E = (E_1,\ldots,E_{2d})$. We thus write
\begin{align*}
  &D_i = E_i = X_i &&\hspace{-7.5em}\text{when }i = 1,\ldots,d\\
  &D_i = E_i = Y_i &&\hspace{-7.5em}\text{when }i = d + 1,\ldots,2d\\
  &D_{2d+1} = S.
\end{align*}
Composition of these are further abbreviated into $D^I$ and $E^I$, with the ordered multi-index $I = (i_1, \ldots, i_n)$, so that $D^I = D_{i_1}\cdots D_{i_n}$ and likewise for $E^I$. The homogeneous length of $I$ is then $|I| = ∑_{j=1}^n μ(i_j)$, where $μ(i)$ is the homogeneous degree of $D_i$, that is $2$ for $i = 2d + 1$ and otherwise $1$. 

We shall often use functions in the Schwartz class $\Schwartz(ℍ^d)$. The usual definition with derivatives replaced with e.\,g.\@ left-invariant vector fields coincides with $\Schwartz(ℝ^{2d + 1})$. The Haar measure on $ℍ^d$ is the Lebesgue measure, and it is invariant under left and right translations. With this measure, a ball $B$ of radius $2^ℓ$ has volume $|B| = C_Q2^{Qℓ} = C_Q2^{(2d + 2)ℓ}$.

It will be crucial for the upcoming analysis that we can perform a form of diagonalisation of the sub-Laplacian. For this we shall need the convolution formula
\[
  f \star g(v) = ∫_{ℍ^d} f(v \go w^{-1})g(w)\de w.
\]
Since the sub-Laplacian is left-invariant, it is given by right convolution with a distribution.  Indeed, any left-invariant operator $T$ for which the mapping $\Schwartz(ℍ^d) \ni f \mapsto Tf(0)$ is continuous is given by right convolution with a tempered distribution. To see this, let $P$ be the parity operator, so that $Pf(w) = f(w^{-1})$, and let $δ$ be the point mass distribution with $(δ,f) = f(0)$. Then $Tf(v) = Tf(v \go 0) = T(f(v\go \argd))(0) = (δ \circ T \mathbin| f(v\go \argd)) = (δ \circ T \circ P \mathbin|_w f(v\go w^{-1})) = f \star (δ \circ T \circ P)(v)$, where we have the usual definition $f \star ω(v) = (ω\mathbin|_w f(v \go w^{-1})) = (ω ∘ P ∘ ℓ_v \mathbin|f)$, with $ℓ_v$ defined by $ℓ_v(f)(w) = f(v \go w)$. The sub-Laplacian is bounded on $\Schwartz(ℍ^d)$, so this applies. It is also essentially self-adjoint and commutes with $P$, so it is natural to write its right convolution kernel $δ \circ ∆ \circ P =: ∆δ$. 

Convolutions are partially diagonalised with the Fourier transform $\F $ on $\Schwartz(ℍ^d)$. This operator is defined such that for each $λ ∈ ℝ\setminus\{0\}$,
\[
  \F f(λ) = ∫_{ℍ^d}f(v)\,U^λ_v\de v 
\]
where $v \mapsto U^λ_v$ is a homomorphism and the $U^λ_v$ are unitary operators on $L^2(ℝ^d)$ defined by 
\[
  U^λ_vα(q) = e^{iλ(s + x·q + x·y/2)}α(q + y).
\]
With the integrated operator above, we mean that for all $α,β ∈ \Schwartz(ℝ^d)$,
\[
  \langle α, \F f(λ) β\rangle_{L^2(ℝ^d)} = ∫_{ℍ^d}f(v)\langle α, U^λ_vβ\rangle_{L^2(ℝ^d)}\de v.
\]
It follows that convolutions of functions $f,g ∈ \Schwartz(ℍ^d)$ satisfy the pointwise composition property
\begin{equation*}\label{diagonaliserad faltning}
  \F (f \star g)(λ) = \F f(λ)\F g(λ).
\end{equation*}

We write $\Schwartz(\hat ℍ^d)$ for the image $\F(\Schwartz(ℍ^d))$. This notation is motivated by the rapid decay and smoothness properties that operators in this image exhibit, see \cite{BahouriCD-Tempered, BahouriCD-Frequency}. On this space, the inverse $\F ^{-1}$ is given by
\[
  f(v) = \frac1{(2π)^{d + 1}}∫_ℝ |λ|^d\,\tr[U_v^{λ*}\F f(λ)] \de λ.
\]

We have the \hypertarget{plancherel}{Plancherel theorem}
\[
  \|f\|_{L^2(ℍ^d)}^2 = \frac1{(2π)^{d + 1}}∫_ℝ|λ|^d\,\|\F f(λ)\|_{\mathup{HS}}^2\de λ,
\]
with $\|\argd\|_\mathup{HS}$ denoting the Hilbert–Schmidt norm of operators on $L^2(ℝ^d)$, and polarisation of this expression leads to the corresponding inner product identity, Parseval's formula
\[
  \inpr{f,g}_{L^2(ℍ^d)} = \frac1{(2π)^{d + 1}}∫_ℝ|λ|^d\,\tr[\F f(λ)^*\F g(λ)]\de λ.
\]

 We write $\Schwartz'(\hat ℍ^d)$ for the space of bounded operators $\Schwartz(\hat ℍ^d) → ℂ$. From Parseval's formula we can define the action of the Fourier transform on distributions on $ℍ^d$: For $ω ∈ \Schwartz'(ℍ^d)$ and $F ∈ \Schwartz(\hat ℍ^d)$, 
we define, with $P$ the parity operator as before,
\[
  (\F ω\mathbin|F) := \big(ω\mathbin{\big|}\overline{\F^{-1}(F^*)}\big) = (ω \circ P \mathbin|\F^{-1}F).
\]

For the distributions we study in this work, $(ω\mathbin|_v\langle α,U^λ_vβ\rangle_{L^2(ℝ^d)})$ is defined for all $λ$ and all $α,β ∈ \Schwartz(ℝ^d)$. This defines $\F ω$ as an operator-valued function of $λ$; $\F ω(λ) = (ω\mathbin|_vU^λ_v)$, for which the definition above reduces to
\[
  (\F ω\mathbin| F) = ∫_ℝ|λ|^d\tr[\F ω(λ)F(λ)]\de λ.
\]
We claim that the $L^2$ norm of such a distribution is
\begin{align*}
  \|ω\|_{L^2(ℍ^d)}^2 &= \frac1{(2π)^{d+1}}∫_ℝ|λ|^d\|\F ω(λ)\|_\mathup{HS}^2\de λ.
\end{align*}
This is so since if one applies the usual Cauchy–Schwarz argument on the frequency side, one finds that
\begin{align*}
  \sup_{\|f\|_{L^2}=1} |(ω\mathbin|f)| &= \sup_{\|f\|_{L^2}=1} |(ω\circ P \mathbin|\F^{-1}\F f)| \\
                                       &= \sup_{\|f\|_{L^2}=1} |(\F ω \mathbin| \F f)| 
                                       ≤ \frac1{(2π)^{d+1}}∫_ℝ|λ|^d\|\F ω(λ)\|_\mathup{HS}^2\de λ,
\end{align*}
and whenever the expression in the right-hand side is finite, the distribution is an $L^2$ function and we have equality here. This also implies the convolution inequality
\[
  \|f \star w \|_{L^∞(ℍ^d)} ≤ \|f\|_{L^2(ℍ^d)}\|ω\|_{L^2(ℍ^d)}.
\]

A key property of the Schrödinger representation of the Heisenberg group is that the sub-Laplacian on $ℍ^d$ is Fourier transformed to a family of harmonic oscillator Hamiltonians. Specifically, we have that, with $α ∈ \Schwartz(ℝ^d)$ and $Vα(q) = |q|^2α(q)$,
\[
  ∆_vU_v^λα = U_v^λ( -λ^2Vα + ∆_{ℝ^d}α) = U_v^λ(-H_λ)α,
\]
where we have put $H_λα = λ^2Vα - ∆_{ℝ^d}α$. Therefore, we find that the convolution kernel of the sub-Laplacian is transformed to $\F (∆δ)$ with
\[
  \F (∆δ)(λ) = ∆_vU^λ_v|_{v=0} = -H_λ.
\]
The operator $H_1$ is the harmonic oscillator Hamiltonian, and has as eigenvectors the Hermite functions $η_m ∈ \Schwartz(ℝ^d)$ with $m ∈ ℕ^d$, which form an orthonormal basis of $L^2(ℝ^d)$ where each $m$ corresponds to the eigenvalue $d + 2|m|$. See \cite{Lebedev} for general properties of these functions, which are products of the Hermite polynomials and the weight $t \mapsto e^{-t^2/2}$. In our setting we shall need the stretched Hermite functions 
\[
  η_m^λ = |λ|^{d/4}η_m(|λ|^{1/2}\argd). 
\]
It is clear that these also form an orthonormal basis, and that
\[
  H_λη_m^λ = |λ|(d + 2|m|)η_m^λ.
\]
Therefore, if we write
\[
  P_m^λ = η^λ_m\Rinpr{η^λ_m,\argd}
\]
for the projection operators of the spectral decomposition of $H_λ$, we have
\begin{align*}
  -∆f = f \star (-∆)δ &= \F ^{-1}\big[\F f\,H_{\ud λ}\big] \\
  &=\F ^{-1}\left[\F f∑_{m ∈ ℕ^d}|\ud λ|(d + 2|m|)P_m^{\ud λ}\right].
\end{align*}
For a polynomial function $p$, it is likewise clear that
\begin{align*}
  p(-∆)f = f \star p(-∆)δ &= \F ^{-1}\left[\F f\,p(\F (-∆δ)) \right] = \F ^{-1}\left[\F f\,p(H_{\ud λ}) \right]\\
        &= \F ^{-1}\left[\F f∑_{m ∈ ℕ^d} p\big(|\ud λ|(d + 2|m|)\big)P^{\ud λ}_m \right].
\end{align*}
It is therefore natural to make the following extension.
\begin{definition}\label{def multiplikator}
For a locally integrable function $φ : [0,∞)  → ℂ$, the operator $φ(-∆)$ acts on $f ∈ \Schwartz(ℍ^d)$ with
\[
  φ(-∆)f = \F ^{-1}\Bigg[\F f ∑_{m ∈ ℕ^d} φ\big(|\ud λ|(d + 2|m|)\big)P^{\ud λ}_m \Bigg].
\]

\end{definition}
Notice that the expression on the right hand side is manifestly left-invariant, and that the correspondence $φ \mapsto φ(-∆)$ is a homomorphism, in the sense that $(ψφ)(-∆) = ψ(-∆)φ(-∆)$ holds formally. Therefore, for $f ∈ \Schwartz(ℍ^d)$ and any nonnegative integer $N$, we have from the \hyperlink{plancherel}{Plancherel theorem} and a Cauchy–Schwarz argument that
\begin{align*}
  &|φ(-∆)f(0)| = |φ(-∆)(1 - ∆)^{-N}(1 - ∆)^Nf(0)|\\
  &= \Bigg|\frac1{(2π)^{d+1}}∫_ℝ|λ|^d\,\tr\Bigg[\F [(1 - ∆)^Nf]∑_{m∈ℕ^d}\frac{φ(|λ|(d+2|m|))}{(1 + |λ|(d + 2|m|))^N}P^λ_m\Bigg]\de λ\Bigg|\\
  &≤ \|(1-∆)^Nf\|_{L^2(ℍ^d)}\sqrt{\frac1{(2π)^{d+1}}∫_ℝ|λ|^d∑_{m∈ℕ^d}\frac{|φ(|λ|(d+2|m|)|^2}{(1 + |λ|(d + 2|m|))^{2N}} \de λ},
\end{align*}
and if $φ$ grows no faster than a power $M$, the integral in $λ$ can be estimated by
\[
  ∫_0^∞λ^d∫_d^∞ν^{d-1}\frac{(1 + λν)^{2M}}{(1 + λν)^{2N}}\de ν\de λ = ∫_d^∞ν^{-2}∫_0^∞λ^d(1 + λ)^{2M-2N}\de λ\de ν < ∞,
\]
for sufficiently large $N$. Furthermore, by the left-invariance, $φ(-∆)f$ is well defined everywhere in $ℍ^d$. We can hence conclude that
\begin{lemma}
  If $φ$ has no more than polynomial growth,  $φ(-∆)f = f \star φ(-∆)δ$ for a tempered distribution $φ(-∆)δ := δ \circ φ(-∆) \circ P$. 
\end{lemma}

Because of its relation to the sub-Laplacian, we shall refer to the quantity $|λ|(d + 2|m|)$ in the Fourier expansion of a function or convolution kernel as \emph{frequency}, and thus speak of low and high frequency regions. 

We shall now show a number of useful lemmas about $L^2(ℍ^d)$ boundedness. We start with two basic properties of sub-Laplacian Fourier multipliers. The first shows that any $φ ∈ L^∞([0,∞))$ gives rise to an $L^2(ℍ^d)$-bounded operator.

\begin{lemma}\label{L2-multiplikator}
  For any $f ∈ L^2(ℍ^d)$, $\|φ(-∆)f\|_{L^2(ℍ^d)} ≤ \|φ\|_{L^∞([0,∞))}\|f\|_{L^2(ℍ^d)}$.
\end{lemma}
\begin{proof}
  From the \hyperlink{plancherel}{Plancherel theorem},
  \begin{align*}
    \!\!\!\|φ(-∆)f\|_{L^2(ℍ^d)}^2 &= \frac1{(2π)^{d+1}}∫_ℝ|λ|^d∑_{m∈ℕ^d}|φ(|λ|(d+2|m|))|^2 \|\F f(λ)η^λ_m\|_{L^2(ℝ^d)}^2\!\!\! \\
                            &≤ \|φ\|_{L^∞([0,∞))}\|f\|_{L^2(ℍ^d)}^2.\qedhere
  \end{align*}
\end{proof}
This basic property shows also that whenever $φ$ is bounded, this definition of sub-Laplacian multipliers is identical with the spectral definition of $φ(-∆)$ on $L^2(ℍ^d)$, a fact that follows from the uniqueness of the spectral theorem.

\begin{lemma}\label{kaerna aer L2}
  If $φ: [0,∞) → ℂ$ satisfies $φ(σ) \lesssim (1 + σ)^{-a}$ for $a > Q/4$, then $φ(-∆)δ ∈ L^2(ℍ^d)$. 
\end{lemma}
\begin{proof}
  
  From the $L^2$ norm of distributions and using the monotonicity of $(1 + σ)^{-a}$, Fubini's theorem and a change of variables, we have
\begin{align*}
  \|φ(-∆)δ\|_{L^2(ℍ^d)}^2 &= \frac1{(2π)^{d+1}}∫_ℝ|λ|^d\|\F (φ(-∆)δ)(λ)\|_\mathup{HS}^2\de λ\\
                          &=\frac1{(2π)^{d+1}}∫_ℝ|λ|^d∑_{m∈ℕ^d}|φ(|λ|(d + 2|m|))|^2\de λ\\
                            &\lesssim ∫_0^∞λ^d∑_{m∈ℕ^d}(1 + λ(d + 2|m|))^{-2a}\de λ\\
                          &\lesssim ∫_0^∞λ^d∫_0^∞ν^{d-1}(1 + λ(1 + ν))^{-2a}\de ν\de λ \\
                          &= ∫_0^∞∫_0^∞\frac{λ^d}{(1 + ν)^{d+1}}ν^{d-1}(1 + λ)^{-2a}\de λ\de ν \\
  &\lesssim ∫_0^∞ λ^d(1 + λ)^{-2a} \de λ< ∞.\qedhere
\end{align*}
\end{proof}

A recurring theme in the upcoming analysis will be the replacement of the left-invariant vector fields $X$, $Y$ and $S$ with functions of the sub-Laplacian, which we then can analyse on the frequency side. This process comes in large down to bookkeeping of the results of the \hypertarget{commutation relation}{commutation relations}\label{commutation relation}
\[
  [X_i,Y_j] = -δ_{ij}S,\qquad [X_i, S] = [Y_i, S] = 0
\]
and dealing with the fact that $∆$ is a second order differential operator while the horizontal left-invariant vector fields are of first order. The first result in that direction is the following estimate.

\begin{lemma}\label{derivatuskshs}
  For all multi-indices $I$ and all $f ∈ \Schwartz(ℍ^d)$,
  \begin{align*}
    \|D^If\|_{L^2(ℍ^d)} &\lesssim \|(-∆)^{|I|/2}f\|_{L^2(ℍ^d)}.
  \end{align*}
\end{lemma}
\begin{proof}
  We start by noting that, since $\F (Sf)(λ) = iλ\F f(λ)$, we have from the \hyperlink{plancherel}{Plan}\-cherel theorem that 
  \begin{align*}
    \|Sf\|_{L^2(ℍ^d)}^2 &= \frac1{(2π)^{d+1}}∫_ℝ|λ|^d∑_{m∈ℕ^d}λ^2\|\F f(λ)η^λ_m\|_{L^2(ℝ^d)}^2 \de λ\\
                          &≤\frac1{(2π)^{d+1}}∫_ℝ|λ|^d∑_{m∈ℕ^d}λ^2(d + 2|m|)^2\|\F f(λ)η_m^λ\|_{L^2(ℝ^d)}^2 \de λ\\
                          &=\|(-∆)f\|^2_{L^2(ℍ^d)}.
  \end{align*}
  Therefore, since $S$ commutes with $E$ as well as with $(-∆)^{1/2}$, it suffices to show that $\|E^If\|_{L^2(ℍ^d)} \lesssim \|(-∆)^{|I|/2}f\|_{L^2(ℍ^d)}$. Furthermore, it will be enough to show that for all indices $j ∈ \{1,\ldots, 2d\}$ and nonnegative integers $k ≤ |I| - 1$, we have for all Schwartz $f$ the estimate
  \[
    \|(-∆)^{k/2}E_if\|_{L^2(ℍ^d)} \lesssim \|(-∆)^{(1+k)/2}f\|_{L^2(ℍ^d)}.
  \]
  And to conclude this relation, it will in turn suffice to show the three estimates
  \begin{align}
    \|E_jf\|_{L^2(ℍ^d)} &\lesssim \|(-∆)^{1/2}f\|_{L^2(ℍ^d)} \label{Ef}\\
    \|E_iE_jf\|_{L^2(ℍ^d)} &\lesssim \|(-∆)f\|_{L^2(ℍ^d)} \label{EEf}
  \end{align}
  and if $2 ≤ k ≤ |I| - 1$,
  \begin{equation}\label{LaEf}
    \|(-∆)^{k/2}E_jf\|_{L^2(ℍ^d)} \lesssim \|(-∆)^{k/2-1}E_j(-∆)f\|_{L^2(ℍ^d)} + \|(-∆)^{(k+1)/2}f\|_{L^2(ℍ^d)}.
  \end{equation}
  This is so, since estimate \pref{Ef} provides for the case $k = 0$, while for $k = 1$, we can use the symmetry of $(-∆)^{1/2}$ and estimate \pref{EEf} to deduce that
  \begin{align*}
    \|(-∆)^{1/2}E_jf\|_{L^2(ℍ^d)}^2 &= \big\langle E_jf, (-∆)E_jf\big\rangle_{L^2(ℍ^d)}\\
                                  &= ∑_{i=1}^{2d}\big\langle E_iE_jf, E_iE_jf\big\rangle_{L^2(ℍ^d)}\\
                                  &\lesssim \|(-∆)f\|_{L^2(ℍ^d)}^2.
  \end{align*}
  If $k ≥ 2$, we use estimate \pref{LaEf} several times to find
  \begin{align*}
    \|(-∆)^{k/2}E_jf\|_{L^2(ℍ^d)} &\lesssim  \|(-∆)^{(k+1)/2}f\|_{L^2(ℍ^d)} + \|(-∆)^{k/2-1}E_j(-∆)f\|_{L^2(ℍ^d)}\\[.2em]
                                  \hspace{2.5em}&\hspace{-2.5em}\lesssim  \|(-∆)^{(k+1)/2}f\|_{L^2(ℍ^d)} + \|(-∆)^{k/2-\floor{k/2}}E_j(-∆)^{\floor{k/2}}f\|_{L^2(ℍ^d)},
  \end{align*}
  and if $k$ is even, estimate \pref{Ef} provides the final upper bound on the second term. If $k$ is odd, $k/2-\floor{k/2} = 1/2$ and we use the same reasoning as for $k = 1$ above.

  Now, it is clear that $\|E_jf\|_{L^2(ℍ^d)}^2 ≤ \big\langle f,-∆f\big\rangle_{L^2(ℍ^d)} = \|(-∆)^{1/2}f\|_{L^2(ℍ^d)}^2$ for every $j$, so estimate \pref{Ef} holds.
  For estimate \pref{EEf}, we have firstly that if $E_iE_j$ commute,
  \[
    0 ≤ \big\|E_iE_jf\big\|_{L^2(ℍ^d)} = \big\langle E_i^2f, E_j^2f \big\rangle_{L^2(ℍ^d)},
  \]
  while from the commutation relation of $X_j$ and $Y_j$ we find that
  \begin{align*}
    \|X_jY_jf\|_{L^2(ℍ^d)}^2 &= \big\langle X_j^2f, Y_j^2f\big\rangle_{L^2(ℍ^d)} + 2\big\langle Sf,X_jY_jf\big\rangle_{L^2(ℍ^d)}\\
            &= \big\langle Y_j^2f, X_j^2f\big\rangle_{L^2(ℍ^d)} - 2\big\langle Sf,Y_jX_jf\big\rangle_{L^2(ℍ^d)}\\
                             &= \sfrac12\big\langle X_j^2f,Y_j^2f\big\rangle_{L^2(ℍ^d)} + \sfrac12\big\langle Y_j^2f,X_j^2f\big\rangle_{L^2(ℍ^d)} - \|Sf\|_{L^2(ℍ^d)}^2 \\
    &= \|Y_jX_jf\|_{L^2(ℍ^d)}^2.    
  \end{align*}
  The nonnegativity of this expression means in particular that
  \[
    \big\langle X_j^2f,Y_j^2f\big\rangle_{L^2(ℍ^d)} + \big\langle Y_j^2f,X_j^2f\big\rangle_{L^2(ℍ^d)} ≥ \|Sf\|_{L^2(ℍ^d)}^2 ≥ 0.
  \]
  It follows that for every $i, j$,
  \begin{align*}
    &|(-∆)f\|_{L^2(ℍ^d)}^2 =\\[.2em]
    &\ANDRA ∑_{|k - ℓ| \neq d} \big\langle E_k^2f,E_ℓ^2f\big\rangle_{L^2(ℍ^d)}
                           +\ ∑_k \Big[\big\langle X_k^2f,Y_k^2f\big\rangle_{L^2(ℍ^d)} + \big\langle Y_k^2f, X_k^2f\big\rangle_{L^2(ℍ^d)}\Big]\\[.2em]
                           &≥ \sfrac12\Big[\big\langle E_i^2f,E_j^2f\big\rangle_{L^2(ℍ^d)} + \big\langle E_j^2f, E_i^2f\big\rangle_{L^2(ℍ^d)}\Big]\\[.2em]
        &≥ \|E_iE_jf\|_{L^2(ℍ^d)}^2.
  \end{align*}
  We show the third estimate \pref{LaEf} for $E_j = X_j$. The other case is similar. Since $∆X_j = X_j∆ + 2Y_jS$, we have
  \begin{align*}
    &\FOERSTA\Hinpr{(-∆)^{k/2}X_jf,(-∆)^{k/2}X_jf} =\\
    &\ANDRA\Hinpr{(-∆)^{k/2}X_jf,(-∆)^{k/2-1}X_j(-∆)f}\\
    &\TERM+ 2\Hinpr{(-∆)^{k/2}X_jf,(-∆)^{k/2-1}Y_jSf} \\[.2em]
    &=\Hinpr{(-∆)^{k/2-1}X_j(-∆)f,(-∆)^{k/2-1}X_j(-∆)f}\\
    &\TERM+ 2\Hinpr{(-∆)^{k/2}X_jf,(-∆)^{k/2-1}Y_jSf} \\
    &\TERM+ 2\Hinpr{(-∆)^{k/2-1}Y_jSf,(-∆)^{k/2-1}X_j(-∆)f}.
  \end{align*}
  To deal with the two last terms, notice that they can be written as two times
  \[
    -\Hinpr{f,X_j(-∆)^{k-1}Y_jSf} + \Hinpr{f,Y_j(-∆)^{k-2}X_j(-∆)Sf}.
  \]
  Now, if we in the first term could push $Y_j$ to the left so that it is to the right of $X_j$, and if we likewise could push $X_j$ in the second term so that it is to the right of $Y_j$, then we would have
  \begin{align*}
    \Hinpr{f,[Y_j,X_j](-∆)^{k-1}Sf} &=  \Hinpr{f,(-∆)^{k-1}S^2f} \\
                                    &= - \|(-∆)^{(k-1)/2}Sf\|_{L^2(ℍ^d)}^2
  \end{align*}
  and we would be done since $\|Sf\|_{L^2(ℍ^d)} ≤ \|∆f\|_{L^2(ℍ^d)}$. When we thus commute $Y_j$ and $X_j$ with the powers of $(-∆)$, we get additional terms with an extra operator $S$. With $|N| = |L| = k - 1$, these terms are of the form
  \begin{align*}
    \Hinpr{f,E^NE^LS^2f} &= (-1)^{k-1}\Hinpr{E^NSf,E^LS^2f}\\
                           &\lesssim \|E^NSf\|_{L^2(ℍ^d)}^2 + \|E^LSf\|_{L^2(ℍ^d)}^2,
  \end{align*}
  and since $|N|,|L| = k-1 < |I| - 1$, we can use the argument for lower $|I|$ to show that $\|E^LSf\|_{L^2(ℍ^d)} \lesssim \|(-∆)^{|L|/2}Sf\|_{L^2(ℍ^d)} \lesssim \|(-∆)^{(k+1)/2}f\|_{L^2(ℍ^d)}$ and likewise for the other term.
\end{proof}

In the proof of Theorem \ref{allmaenna s}, an important step will be to estimate expressions of the form $\||\argd|^{2N}D^If\|_{L^2(ℍ^d)}$. The following lemma is the first step in that endeavour. The special polynomial $is - \sfrac14|z|^2$, which has a length comparable to $|v|^2$, shall turn out to be key when we translate position multipliers into frequency derivatives in the next section.

\begin{lemma}\label{stora polynomderivatshs} 
  For a homogeneous polynomial $p$ of degree $m$, it holds for every multi-index $I$ and Schwartz $f$ that, if $m + |I|$ is even,
  \begin{align*}
    \hspace{1em}&\hspace{-1em}\Hnorm{pD^If}^2 ≤\\
    &C_{p,I} ∑_{\substack{\ceil{|I|/2} - \ceil{m/2} ≤ j ≤ \ceil{|I|/2}\\j≥0}}\left\|(i\ud s - \sfrac14|\ud z|^2)^{\ceil{m/2} - \ceil{|I|/2} + j}(-∆)^jf\right\|_{L^2(ℍ^d)}^2,  \end{align*}
  while if $m + |I|$ is odd,
  \begin{align*}
    \hspace{1em}&\hspace{-1em}\Hnorm{pD^If}^2 ≤\\
    &C_{p,I} ∑_{\substack{\floor{|I|/2} - \ceil{m/2} + 1 ≤ j ≤ \floor{|I|/2}\\j≥0}}\Big|\Big\langle(i\ud s - \sfrac14|\ud z|^2)^{\ceil{m/2} - \floor{|I|/2} + j - 1}(-∆)^jf, \\[-1em]
    &\hspace{12em}(i\ud s - \sfrac14|\ud z|^2)^{\ceil{m/2} - \floor{|I|/2} + j}(-∆)^jf\Big\rangle_{L^2(ℍ^d)}\Big|\\
    +\ &C_{p,I}∑_{\substack{(|I|-m-1)/2≤j≤\ceil{|I|/2}-1\\j≥0}}\Big| \Big\langle(i\ud s - \sfrac14|\ud z|^2)^{(m-|I|+1)/2+j}(-∆)^jf,\\[-1em]
    &\hspace{12em}(i\ud s - \sfrac14|\ud z|^2)^{(m-|I|+1)/2+j}(-∆)^{j+1}f\Big\rangle_{L^2(ℍ^d)}\Big|.
  \end{align*}
\end{lemma}

For the proof of Lemma \ref{stora polynomderivatshs}, it will be convenient to have the following minor result at hand.

\begin{lemma}\label{polynomderivatshs} 
  For a homogeneous polynomial $p$ on $ℍ^d$ of degree $m$ and multi-index $I$ we have for all $f ∈ \Schwartz(ℍ^d)$ that
  \begin{align*}
    \big\|pD^If\big\|_{L^2(ℍ^d)}^2 &\lesssim \sfrac12\Hinpr{p(-∆)^{\lfloor|J|/2\rfloor}f,\,p(-∆)^{\lceil|J|/2\rceil}f} + \text{c.\,c.}\\[.4em]
    &\TERM+ ∑_n∑_{|J|<|I|}\Hinpr{q_J^n(-∆)^{\lfloor|J|/2\rfloor}f,\,q_J^n(-∆)^{\lceil|J|/2\rceil}f} + \text{c.\,c.},
  \end{align*}
  where c.\,c.\@ signifies the complex conjugate of the preceding term, the sum in $n$ is finite and each $q_J^n$ is a homogeneous polynomial of degree $m - |I| + |J|$.
\end{lemma}

When making use of this lemma, we shall suppress the sum in $n$, so that we can write
\[
  \big\|pD^If\big\|_{L^2(ℍ^d)}^2 \lesssim ∑_{|J|≤|I|}\Hinpr{q_J(-∆)^{\lfloor|J|/2\rfloor}f,\,q_J(-∆)^{\lceil|J|/2\rceil}f} + \text{c.\,c.}
\]
for homogeneous polynomials $q_J$ of degree $m - |I| + |J|$ such that $q_J = p$ when $|J| = |I|$.

\begin{proof}
  Notice firstly that it will suffice to show this in the special case when $p$ is real, since if $p = p^\textsc r + ip^\textsc i$ for real polynomials $p^\textsc r$ and $p^\textsc i$,
  \[
    \|pD^If\|_{L^2(ℍ^d)}^2 = \|p^\textsc rD^If\|_{L^2(ℍ^d)}^2 + \|p^\textsc iD^If\|_{L^2(ℍ^d)}^2
  \]
  and we have likewise that
  \begin{align*}
    &\Hinpr{p^\textsc r(-∆)^{\lfloor|I|/2\rfloor}f,\,p^\textsc r(-∆)^{\lceil|I|/2\rceil}f} + \Hinpr{p^\textsc i(-∆)^{\lfloor|I|/2\rfloor}f,\,p^\textsc i(-∆)^{\lceil|I|/2\rceil}f} \\
    &=\Hinpr{p(-∆)^{\lfloor|I|/2\rfloor}f,\,p(-∆)^{\lceil|I|/2\rceil}f}.
  \end{align*}
  In the following, we therefore assume that $p$ is real. We have for some homogeneous polynomials $q_J$ of degrees $m - |I| + |J|$ that
  \[
    D^I(pf) = pD^If + ∑_{|J|<|I|}q_JD^Jf.
  \]
  Therefore  
  \begin{align*}
    \Hnorm{D^I(pf)}^2
    &= \Hnorm{pD^If}^2 + ∑_{|J|<|I|}∑_{|K|<|I|}\Hinpr{q_JD^Jf, q_KD^Kf}\\
    &\TERM+ ∑_{|J|<|I|}\Hinpr{q_JD^Jf, pD^If} + \text{c.\,c.}\\[.4em]
    &= \Hnorm{pD^If}^2 + ∑_{|J|<|I|}∑_{|K|<|I|}\Hinpr{q_JD^Jf, q_KD^Kf}\\
    &\TERM+ ∑_{|J|<|I|}\Hinpr{q_JD^Jf, D^I(pf)} + \text{c.\,c.}\\
    &\TERM- 2∑_{|J|<|I|}∑_{|K|<|I|}\Hinpr{q_JD^Jf, q_KD^If}\\ 
    &= \Hnorm{pD^If}^2 - ∑_{|J|<|I|}∑_{|K|<|I|}\Hinpr{q_JD^Jf, q_KD^Kf}\\
    &\TERM+ ∑_{|J|<|I|}\Hinpr{q_JD^Jf, D^I(pf)} + \text{c.\,c.}
  \end{align*}
  Hence
  \begin{align}
    \Hnorm{pD^If}^2 
    &= \Hnorm{D^I(pf)}^2 + ∑_{|J|<|I|}∑_{|K|<|I|}\Hinpr{q_JD^Jf, q_KD^Kf}\notag\\
    &\TERM- ∑_{|J|<|I|}\Hinpr{q_JD^Jf, D^I(pf)} + \text{c.\,c.}\notag\\
    &\lesssim \Hnorm{D^I(pf)}^2 + ∑_{|J|<|I|}\Hnorm{q_JD^Jf}^2\notag\\
    &\lesssim \Hnorm{(-∆)^{|I|/2}(pf)}^2 + ∑_{|J|<|I|}\Hnorm{q_JD^Jf}^2,\label{pDf}
  \end{align}
  where in the last estimate we used Lemma \ref{derivatuskshs}.

  Now if $|I|$ is even, we have $(-∆)^{|I|/2}(pf) = p(-∆)^{|I|/2}f + ∑_{|J|<|I|}q'_JD^Jf$ for some new homogeneous polynomials $q_J'$ of degrees $m - |I| + |J|$. Hence
  \begin{align*}
    \big\|(-∆)^{|I|/2}(pf)\big\|_{L^2(ℍ^d)}^2 &=  \big\|p(-∆)^{|I|/2}f\big\|_{L^2(ℍ^d)}^2 \\
                                              &\TERM+\ ∑_{|J|<|I|} \big\langle p(-∆)^{|I|/2}f,q_J'D^Jf\big\rangle_{L^2(ℍ^d)} + \text{c.\,c.}\\
                                              &\TERM+\ ∑_{|J|<|I|} ∑_{|K|<|I|}\big\langle q_J'D^Jf,q_K'D^Kf\big\rangle_{L^2(ℍ^d)}\\
                                                &\lesssim \big\|p(-∆)^{|I|/2}f\big\|_{L^2(ℍ^d)}^2 + ∑_{|J|<|I|}\Hnorm{q_J'D^Jf}^2.
  \end{align*}
  We can then combine this estimate with the one above, and conclude that
  \begin{align*}
    \Hnorm{pD^If}^2 &\lesssim \Hnorm{p(-∆)^{|I|/2}f}^2 + ∑_{|J|<|I|}\Hnorm{q_JD^Jf}^2\\
    &\TERM+ ∑_{|J|<|I|}\Hnorm{q'_JD^Jf}^2,\\[-2.2em]
  \end{align*}
  as requested.

  If $|I|$ is odd, we have 
  \[
    (-∆)^{\floor{|I|/2}}(pf) = p(-∆)^{\floor{|I|/2}}f + ∑_{|J|<2\floor{|I|/2}}π_JD^Jf,
  \]
   for some homogeneous polynomials $π_J$ of degree $m - 2\floor{|I|/2} + |J|$, and therefore 
  \begin{align*}
    \Hnorm{(-∆)^{|I|/2}(pf)}^2 
    \hspace{-4em}&\hspace{4em}= ∑_{j=1}^{2d}\Hnorm{E_j(-∆)^{\floor{|I|/2}}(pf)}^2\\
    &\lesssim ∑_{j=1}^{2d}\Bigg[\!\Hnorm{E_j\big[p(-∆)^{\floor{|I|/2}}f\big]}^2 + \!∑_{|J|<2\floor{|I|/2}}\Hnorm{E_j\big[π_JD^Jf\big]}^2\!\Bigg]\\
                 &= \Hinpr{(-∆)^{\floor{|I|/2}}(pf),\,(-∆)\big[(-∆)^{\floor{|I|/2}}(pf)\big]}\\
                 &\TERM+\ ∑_{j=1}^{2d}∑_{|J|<2\floor{|I|/2}}\Hnorm{E_jπ_JD^Jf + π_JE_jD^Jf}^2\\
                 &\lesssim  \Hinpr{(-∆)^{\floor{|I|/2}}(pf),\,(-∆)\big[(-∆)^{\floor{|I|/2}}(pf)\big]}\\
    &\TERM+\ ∑_{j=1}^{2d}∑_{|J|<2\floor{|I|/2}}\Bigg[\Hnorm{E_jπ_JD^Jf}^2 + \Hnorm{π_JE_jD^Jf}^2\Bigg].
  \end{align*}
  To deal with the inner product, we use that for any Schwartz $g$ and real polynomial $p$,
  \begin{align*}
    &\FOERSTA 2\Hinpr{pg,\, (-∆)[pg]} =\\[.4em]
    &\ANDRA\Hinpr{pg,\,p(-∆)g} + \text{c.\,c.}- 2∑_{j=1}^{2d}∫_{ℍ^d}\Big(p\overbar gE_jpE_jg + \text{c.\,c.} + |g|^2pE_j^2p\Big)\\
    &=\Hinpr{pg,\,p(-∆)g} + \text{c.\,c.} - 2∑_{j=1}^{2d}∫_{ℍ^d}\Big(E_j\big[|g|^2\big]pE_jp + |g|^2pE_j^2p\Big)\\
    &= \Hinpr{pg,\,p(-∆)g} + \text{c.\,c.} + 2∑_{j=1}^{2d}\|gE_jp\|_{L^2(ℍ^d)}^2,
  \end{align*}
  and this means for $g = (-∆)^{\floor{|I|/2}}f$ that
  \begin{align*}
    &\FOERSTA\Hinpr{p(-∆)^{\floor{|I|/2}}f, (-∆)\big[p(-∆)^{\floor{|I|/2}}f\big]} \\[.4em]
    &= \sfrac12\Hinpr{p(-∆)^{\floor{|I|/2}}f, p(-∆)^{\floor{|I|/2} + 1}f} + \text{c.\,c.} \\
    &\TERM+\ ∑_{j=1}^{2d}\|E_jp(-∆)^{\floor{|I|/2}}f\|_{L^2(ℍ^d)}^2\\
    &\lesssim \sfrac12\Hinpr{p(-∆)^{\floor{|I|/2}}f, p(-∆)^{\floor{|I|/2} + 1}f} + \text{c.\,c.}\\
    &\TERM+\ ∑_{j=1}^{2d}∑_{|J|=\floor{|I|/2}}\|q^j_JD^Jf\|_{L^2(ℍ^d)}^2,
  \end{align*}
  for some homogeneous polynomials $q^j_J$ of order $m - 1 = m - |I| + |J|$.   Inserting this into the estimate above and estimate \pref{pDf}, we see that
  \begin{align*}
    \big\|pD^If\|_{L^2(ℍ^d)}^2 &\lesssim \sfrac12\Hinpr{p(-∆)^{\floor{|I|/2}}f, p(-∆)^{\floor{|I|/2} + 1}f} + \text{c.\,c.} \\
                               &\TERM+\  ∑_{|J|<|I|}\Hnorm{q_JD^Jf}^2 +  ∑_{j=1}^{2d}∑_{|J|=\floor{|I|/2}}\|q^j_JD^Jf\|_{L^2(ℍ^d)}^2\\
    &\TERM+ ∑_{j=1}^{2d}∑_{|J|<2\floor{|I|/2}}\Bigg[\Hnorm{E_jπ_JD^Jf}^2 + \Hnorm{π_JE_jD^Jf}^2\Bigg].
  \end{align*}
  The sums are all of the form $∑_{|J|<|I|}\|q_JD^Jf\|_{L^2(ℍ^d)}$ for polynomials $q_J$ of order $m - |I| + |J|$, so these terms can be treated with induction on $|I|$ to find the requested formula.
\end{proof}

\begin{proof}[Proof of Lemma \ref{stora polynomderivatshs}] 
   Before we apply Lemma \ref{polynomderivatshs}, we shall make use of the fact that, for even $m$, we can assume that $p(z,s) = (is - \frac14|z|^2)^{m/2}$. Thereafter applying this lemma to $pD^If$, we get a sum with polynomials of different degrees.  For the terms with odd $|J|$, we shall use the property that for homogeneous polynomials $q_J$, we can write $q_J^2 = ∑_iq_{J,i}^-q_{J,i}^+$, which is a finite sum where each $q_{J,i}^-$ has homogeneous degree one less than that of $q_J$, while the degree of $q_{J,i}^+$ is one more. This can always be done when the $q_J$ are of odd degree. When the $q_J$ are of even degree, we can do this for all polynomials except of those of the form $q_J(z,s) = s^{(m - |I| + |J|)/2}$, the square of which we cannot refactor into a product of two odd polynomials. 

  We start with treating the case when $|I|$ is even.  We find that there are polynomials $q_J$ of homogeneous degrees $m - |I| + |J|$ such that
\begin{align*}
  &\big\|pD^If\big\|_{L^2(ℍ^d)}^2 \lesssim ∑_{|J|≤|I|}\Hinpr{q_J(-∆)^{\floor{|J|/2}}f,\,q_J(-∆)^{\ceil{|J|/2}}f} + \text{c.\,c.}\\[.4em]
  &=∑_{\substack{|I|-m≤|J|≤|I|\\|J|\,\mathrm{even}}}2\Hnorm{q_J(-∆)^{|J|/2}f}^2\\
  &\TERM+ ∑_{\substack{|I|-m≤|J|<|I|\\|J|\,\mathrm{odd}}}∑_i\Hinpr{q_{J,i}^-(-∆)^{\floor{|J|/2}}f,\,q_{J,i}^+(-∆)^{\ceil{|J|/2}}f} + \text{c.\,c.}\\
  &\TERM+ ∑_{\substack{\text{$m$ odd}\\1≤2j+1<|I|\\m - |I| + 2j + 1 ≥ 0}}C_j\Hinpr{\ud s^{(m-|I|+2j+1)/2}(-∆)^jf,\,\ud s^{(m -|I|+2j+1)/2}(-∆)^{j+1}f}\\
  &\TERM+\text{c.\,c.}
\end{align*}
The even $|J|$ terms have polynomials of degree $m - |I| + |J|$, and therefore they are for even $m$ estimated from above by
\[
  \Hnorm{(i\ud s - \sfrac14|\ud z|^2)^{(m-|I|+|J|)/2}(-∆)^{|J|/2}f}^2.
\]
If $m$ is odd, we have polynomials of odd degree. To treat these we shall have use of the following claim:
For a homogeneous polynomial $p$ of odd degree $2n + 1$ and $g ∈ \Schwartz(ℍ^d)$,
\begin{equation}\label{udda polynom}
  \|pg\|_{L^2(ℍ^d)}^2 \lesssim \Big|\Hinpr{(i\ud s - \sfrac14|\ud z|^2)^ng,\,(i\ud s - \sfrac14|\ud z|^2)^{n+1}g}\Big|.
\end{equation}
To see this, notice first that we can assume that $p$ is monomial. Then $p(z,s) = \sfrac12z_iq(z,s)$ for some $i ∈ \{1,\ldots,2d\}$ and some monomial $q$ of degree $2m$. Hence
\begin{align*}
  \Hnorm{pg}^2
  &= ∫_{ℍ^d}\sfrac14z_i^2|q(v)|^2|g(v)|^2\de v \lesssim  ∫_{ℍ^d}\sfrac14z_i^2\big|is - \sfrac14|z|^2\big|^{2n}|g(v)|^2\de v\\
  &=\Hnorm{\sfrac12\ud z_i(i\ud s - \sfrac14|\ud z|^2)^ng}^2 ≤ ∑_{i=1}^{2d} \Hnorm{\sfrac12\ud z_i(i\ud s - \sfrac14|\ud z|^2)^ng}^2\\
  &=\Hinpr{(i\ud s - \sfrac14|\ud z|^2)^ng,\, (\sfrac14|\ud z|^2 - i\ud s)(i\ud s - \sfrac14|\ud z|^2)^ng} + \text{c.\,c.},
\end{align*}
as requested. We therefore have, with $n = (m - |I| + |J|-1)/2$,
\begin{align*}
  &2\Hnorm{q_J(-∆)^{|J|/2}f}^2 \lesssim \\
  &\qquad\Big|\Hinpr{(i\ud s - \sfrac14|\ud z|^2)^n(-∆)^{|J|/2}f,\,(i\ud s - \sfrac14|\ud z|^2)^{n+1}(-∆)^{|J|/2}f}\Big|\\[.4em]
  &= \Big|\Big\langle(i\ud s - \sfrac14|\ud z|^2)^{\floor{(m - |I| + |J|)/2}}(-∆)^{|J|/2}f,\\
  &\hspace{4em}(i\ud s - \sfrac14|\ud z|^2)^{\ceil{(m - |I| + |J|)/2}}(-∆)^{|J|/2}f\Big\rangle_{L^2(ℍ^d)}\Big|.
\end{align*}

In the terms with odd $|J|$, we have polynomials $q_J^-$ and $q_J^+$ of homogeneous degrees $m - |I| + |J| - 1$ and $m - |I| + |J| + 1$, respectively. When $m$ is even, so are the polynomials, and we have
\begin{align}
  &\FOERSTA\Hinpr{q_J^-(-∆)^{\floor{|J|/2}}f,\, q_J^+(-∆)^{\ceil{|J|/2}}f} + \text{c.\,c.}\notag\\
  &\lesssim \Hnorm{q_J^-(-∆)^{\floor{|J|/2}}f}^2 + \Hnorm{q_J^+(-∆)^{\ceil{|J|/2}}f}^2\label{udda J}\\
  &\lesssim \Hnorm{(i\ud s - \sfrac14|\ud z|^2)^{(m - |I| + |J| - 1)/2}(-∆)^{(|J|-1)/2}f}^2\notag\\
  &\TERM + \Hnorm{(i\ud s - \sfrac14|\ud z|^2)^{(m -|I| + |J| + 1)/2}(-∆)^{(|J|+1)/2}f}^2.\notag    
\end{align}
When $m$ is odd, we use this reasoning together with the odd polynomial claim \pref{udda polynom}. These are of the same kind of expressions as the case with even $|J|$.

For the third sum, we claim that, for any natural number $n = (m - |I|+2j+1)/2$,
\[
  s^{2n} = \big(s + i\sfrac14|z|^2\big)^{2n} + ∑_i q_i^-q_i^+,
\]
for some homogeneous polynomials $q_i^-$, $q_i^+$ of degrees $4n-1$ and $4n+1$. This is clear from the expansion of $(s + i\sfrac14|z|^2)^{2n}$, since all terms except $s^{2n}$ include first degree factors. The inner products coming from the leading term are then of the requested form
\begin{align*}
  &\Hinpr{(i\ud s - \sfrac14|\ud z|^2)^{(m-|I|+2j+1)/2}(-∆)^jf,\,(i\ud s - \sfrac14|\ud z|^2)^{(m-|I|+2j+1)/2}(-∆)^{j+1}f}\\
  &\hspace{14em}+\ \text{c.\,c.},
\end{align*}
and the rest terms can be treated as the other ones with odd $|J|$.

We continue with the cases with odd $|I|$. We have as before, using Lemma \ref{polynomderivatshs}, that there are polynomials $q_J$ of homogeneous degrees $m - |I| + |J|$ and $q_{J,i}^-$, $q_{J,i}^+$ of degree one less and one more, such that
\begin{align*}
  & \big\|pD^If\big\|_{L^2(ℍ^d)}^2 \lesssim \\[.4em]
  &\ANDRA∑_{|J|≤|I|}\Hinpr{q_J(-∆)^{\floor{|J|/2}}f,\,q_J(-∆)^{\ceil{|J|/2}}f} + \text{c.\,c.}\\
  &=∑_{\substack{|I|-m<|J|<|I|\\|J|\,\mathrm{even}}}2\Hnorm{q_J(-∆)^{|J|/2}f}^2\\
  &\TERM+ ∑_{\substack{|I|-m<|J|≤|I|\\|J|<|I|\text{ if $m$ even}\\|J|\,\mathrm{odd}}}∑_i\Hinpr{q_{J,i}^-(-∆)^{\floor{|J|/2}}f,\,q_{J,i}^+(-∆)^{\ceil{|J|/2}}f} + \text{c.\,c.}\\
  &\TERM+ ∑_{\substack{\text{\rule{0pt}{.7em}$m$ even}\\1≤2j+1<|I|\\ m - |I| + 2j + 1≥0}}C_j\Hinpr{\ud s^{(m-|I|+2j+1)/2}(-∆)^jf,\,\ud s^{(m+|I|+2j+1)/2}(-∆)^{j+1}f}\\
  &\TERM+\text{c.\,c.}\\
  &\TERM+ \text{for even $m$:}\ \Hinpr{p(-∆)^{\floor{|I|/2}}f,\,p(-∆)^{\ceil{|I|/2}}f} + \text{c.\,c.}
\end{align*}
This case is similar to that of even $|I|$, and the same techniques can be applied. The only essential difference from the even case can be found in the last line, where the leading terms are written for even $m$. If we would apply estimates \pref{udda J} to these terms with $|J| = |I|$, we would get a polynomial of power larger than $m$, and this would be a problem later on. But by our assumptions in the beginning of the proof, these terms are already of the requested form. 
\end{proof} 

We end this preparatory section with a version of Taylor's theorem for the Heisenberg group. So far we have only considered left-invariant vector fields $X,Y,S,E,D$, but there are of course also right-invariant counterparts. We shall denote these $\tilde X, \tilde Y, \tilde E, \tilde D$, while $S$ is both left- and right-invariant. As all right-invariant operators, they are given by left-convolution with a distribution, and they commute with left-invariant operators. 
For the upcoming claim, we introduce the following notion.
\begin{definition}
  The left- and right-invariant \emph{homogeneous directional derivatives} of degree $n$ in direction $v ∈ ℍ^d$ of a function $f : ℍ^d → ℂ$ evaluated at $w ∈ ℍ^d$ are, respectively,
  \[
    M^n_vf(w) = ∂^n_t f(w \go δ_t(v))\big|_{t=0},\quad   \tilde M^n_vf(w) = ∂^n_t f(δ_t(v) \go w)\big|_{t=0}.
  \]
\end{definition}
One could also reasonably call these objects the left and right Taylor polynomial terms, as the following theorem shows. It is the analogue of Taylor's theorem on Euclidean space. There is also a theorem in \cite{FollandS} denoted Taylor's theorem for homogeneous groups. That result is however not explicit, as it only gives a rough estimate of the error term, which does not suffice for our purpose.

\begin{theorem}\label{Taylors s}
  $M^n_v$ and $\tilde M^n_v$ are differential operators of homogeneous degree $n$, and for any $f : ℍ^d → ℂ$ such that $M^n_vf(w)$ or $\tilde M^n_vf(w)$ exist at $w ∈ ℍ^d$, these expressions are as functions of $v$ polynomials of homogeneous degree $n$.

  Furthermore, for all $n ∈ ℕ$, $w,v ∈ ℍ^d$ and $f ∈ \Continuous^{n+1}(ℍ^d)$ we have
  \begin{align*}
    f(w \go v) &= ∑_{j=0}^n \frac1{j!} M_v^jf(w) + \frac1{n!}∫_0^1(1-t)^n∂_t^{n+1}f(w \go δ_t(v))\de t\\
               &= ∑_{j=0}^n \frac1{j!} \tilde M_w^jf(v) + \frac1{n!}∫_0^1(1-t)^n∂_t^{n+1}f(δ_t(w) \go v)\de t.
  \end{align*} 
\end{theorem}
\begin{proof}
  We show this only for the left-invariant expression, since the other is analogous. To see that $M^n_v$ is a differential operator of homogeneous degree $n$, we notice that
\begin{align*}
  \hspace{2em}&\hspace{-2em}∂_tf(v' \go δ_t(v)) =\\
                        &z·∇_{ℝ^{2d}}f(v' \go δ_t(v)) + \big(\sfrac12(y'·x - x'·y) + 2ts\big)Sf(v' \go δ_t(v))\\
  &= z·Ef(v' \go δ_t(v)) + 2ts\,Sf(v' \go δ_t(v)),
\end{align*}
where by $∇_{ℝ^{2d}}$ we mean the gradient in the first $2d$ variables. Consequently, if we keep in mind that $S$ commutes with the vector fields in $E$, we find that there are nonnegative integers $C_{a,b}$ such that
\[
  ∂_t^nf(w \go δ_t(v)) = ∑_{\substack{a+b ≤ n\\[.1em] 0≤b≤a}} C_{a,b}\,t^{a-b}(sS)^a(z·E)^{n-a-b}f(w \go δ_t(v)).
\]
Letting $t = 0$ picks out the part where $a=b$, so that
\[
  M^n_vf(w) = ∑_{a = 0}^{\lfloor  n/2\rfloor}C_a (sS)^a(z·E)^{n-2a}f(w),
\]
from which it is clear that $M^n_v$ is a differential operator of homogeneous degree $n$. In particular, if we let $e_i$ with $i = 1,\ldots,2d+1$ denote the coordinate unit vectors, we have
\begin{align*}
  M_{e_i}^n &= E_{e_i}^n\quad \text{for }i = 1,\ldots, 2d\\
  M_{e_{2d+1}}^{2n} &= C_nS^n\\
  M_{e_{2d+1}}^{2n+1} &= 0.
\end{align*}

As for the dependence on $v$, it is easy to see that $M^n_v$ is homogeneous of degree $n$. To see that it is polynomial, we shall show that the action of left-invariant vector fields of a sufficiently large homogeneous degree, $n + 1$, will result in the zero function. 

We see from the last equations above that we can use the $M_v^n$ form to express the left-invariant vector fields. We therefore let $u$ be a coordinate unit vector and study the homogeneous directional derivative of degree $j$ in the direction $v'$ of $v \mapsto M^n_vf(w)$. This derivative is
\begin{align*}
  &∂_t^jM^n_{v\,\go\, δ_t(v')}f(w)\big|_{t=0} =\\[.4em]
  &\ANDRA ∂_t^j\Bigg[∑_{a=0}^{\floor{n/2}}[(v\go δ_t(v'))_sS]^a[(v\go δ_t(v'))_z·E]^{n-2a}f(w)\Bigg]\Bigg|_{t=0}\\
  &= ∂_t^j\Bigg[∑_{a=0}^{\floor{n/2}}(s + s't^2 + \sfrac12(y·x' - x·y')t)^aS^a\big[[(z + tz') · E_i]^{n-2a}f\big](w)\Bigg]\Bigg|_{t=0},
\end{align*}
where the $j$th derivative of each term vanishes if $j > 2a + n - 2a = n$.

Finally, we have by repeated integration by parts that
\begin{align*}
  \hspace{-2em}\hspace{2.5em}&\hspace{-2.5em}\frac1{n!}∫_0^1(1-t)^n∂_t^{n+1}f(w \go δ_t(v))\de t =\\[-.5em]
                        &\frac1{n!}(1-t)^n\,∂_t^nf(w \go δ_t(v))\Big|_{t=0}^1 + \frac1{(n-1)!}∫_0^1(1-t)^{n-1}∂^n_tf(w \go δ_t(v))\de t\hspace{-2em}\\[.3em]
                        &= - ∑_{j=0}^n \frac1{j!}∂^j_tf(w \go δ_t(v))\Big|_{t=0} + f(w \go v).\qedhere
\end{align*}
\end{proof}

\section{Diagonalisation of the kernel $φ(-∆)δ$}\label{nyckeluppskattningsavdelning}
We saw in Lemma \ref{kaerna aer L2} that for $φ$ with sufficiently rapid decay, $φ(-∆)$ has a right convolution kernel that is a function. This kernel has the form
\begin{align*}
  φ(-∆)δ(v) &= \F ^{-1}\left[∑_{m ∈ ℕ^d} φ\big(|\ud λ|(d + 2|m|)\big)\,P^{\ud λ}_m\right](v)\\
          &= \frac1{(2π)^{d+1}}∫_ℝ |λ|^d ∑_{m ∈ ℕ^d} φ(|λ|(d + 2|m|)) \inpr{U^λ_vη^λ_m,\,η^λ_m}_{L^2(ℝ^d)}\de λ.
\end{align*}

The formula can be simplified if we use the following well known identity, see  e.\,g.\@ Theorem {\capfigures 1.3.4} of \cite{Thangavelu-Lectures}.

\begin{lemma}\label{Laguerrehs}
  \[
    ∑_{|m|=n}\inpr{U^λ_{z,s}η^λ_m,\,η_m^λ}_{L^2(ℝ^d)} = e^{-iλs}\, L^{d-1}_n\big(|λ||z|^2/2\big)\,e^{-|λ||z|^2/4},
  \]
  with $L^{d-1}_n$ the Laguerre polynomial of order $n$ and type $d-1$.
\end{lemma}
We refer to \cite{Lebedev} and \cite{Thangavelu-Lectures} for general properties of the Laguerre polynomials. 
For rapidly decreasing $φ$, we can therefore write the right convolution kernel of $φ(-∆)$ as
\begin{align*}
  &φ(-∆)δ(z,s) = \\
  &\ANDRA\frac1{(2π)^{d+1}}∫_ℝ |λ|^d e^{-iλs}∑_{n=0}^∞ φ(|λ|(d + 2n))\,L^{d-1}_n(|λ||z|^2/2)e^{-|λ||z|^2/4}\de λ.
\end{align*}

The Laguerre polynomials form an orthogonal basis of $L^2(ℝ_+, w)$ with the weight $w(t) = t^{d-1}e^{-t}$. Namely, we have
\begin{align*}
  \frac{(d-1 + n)!}{n!}\,δ_{nm} \hspace{-.5em}&\hspace{.5em}= ∫_0^∞ L^{d-1}_n(t)L^{d-1}_m(t)\,t^{d-1}e^{-t}\dd t\\
                                  &= \frac1{2^{d-1}}∫_0^∞ L^{d-1}_n(|λ|r^2/2)L^{d-1}_m(|λ|r^2/2)e^{-|λ|r^2/2}r^{2d-1} |λ|^d\de r\\
                                  &= \frac{(d-1)!}{(2π)^d}|λ|^d ∫_{ℝ^{2d}} L^{d-1}_n(|λ||z|^2/2)L^{d-1}_m(|λ||z|^2/2)e^{-|λ||z|^2/2} \de z.
\end{align*}
We say that a function $f : ℍ^d → ℂ$ with $f(z,s) = \tilde f(|z|,s)$ for some $\tilde f ∈ L^2(ℝ_+×ℝ)$ is \emph{biradial}. For such an $f$, we conclude from the above that we have the decomposition
\begin{equation}\label{G}
  f = \frac1{(2π)^{d+1}}∫_ℝ |λ|^d ∑_{n=0}^∞ \combin{d - 1 + n}n\, \inpr{ℓ_n^λ,f}_{L^2(ℍ^d)}\, ℓ_n^λ\de λ =: \G ^{-1}\G f
\end{equation}
with the operator $\G$ defined from
\[
  \G f(n,λ) = \inpr{ℓ_n^λ,f}_{L^2(ℍ^d)}
\]
and
\[
  ℓ_n^λ(z,s) = \combin{d-1+n}n^{-1}\,e^{-iλs - |λ||z|^2/4}\,L^{d-1}_n(|λ||z|^2/2).
\]
In more generality, we say that a distribution $ω ∈ \Schwartz'(ℍ^d)$ is biradial if there is a map $Ω : ℕ × ℝ → ℂ$ such that
\[
  \F ω(λ) = ∑_{m ∈ ℕ^d} Ω(|m|,λ)P_m^λ.
\]
For these, we put $\G ω = Ω$. Examples are the kernels $φ(-∆)δ$, for which
\[
  \G[φ(-∆)δ](n,λ) = φ(|λ|(d + 2n)).
\]
This definition coincides with the definition of biradiality and that of $\G$ for functions, since if $ω$ is a function, this means that
\[
  \F ω(λ) = ∑_{m∈ℕ^d}Ω(|m|,λ)P_m^λ = ∑_{n=0}^∞\Gω(n,λ)∑_{|m|=n}P^λ_m,
\]
and if we apply $\F^{-1}$ to both sides, Lemma \ref{Laguerrehs} and identity \pref G show that the two definitions are the same. In particular, every biradial distribution is diagonal in the $η^λ_m$ basis, with equal coefficients for equal length of $m$.
A straight-forward computation shows that any distribution $ω$ such that $\F ω$ is diagonal in the $η_m^λ$ basis satisfies the dilation identity $\F^{-1}[\F ω(t^2\ud λ)] = t^{-Q}ω\circ δ_{1/t}$ for all nonzero $t$, which can be compared to the dilation behaviour in Euclidean space. In particular,  biradial distributions satisfy $\G^{-1}[\G ω(\ud n,t^2\ud λ)] = t^{-Q} ω\circ δ_{1/t}$, and for the kernels of $φ(-∆)$, this identity becomes
\begin{equation}\label{multiplikatorutdragning2}
  φ(t^2(-∆))δ = t^{-Q}φ(-∆)δ \circ δ_{1/t}.    
\end{equation}

Of particular interest to us will also be that the Plancherel theorem for biradial functions becomes the \hypertarget{biradial}{diagonal expression}
\begin{equation*}\label{Plancherel biradiell}
  \|f\|_{L^2(ℍ^d)}^2 = \frac1{(2π)^{d+1}} ∫_ℝ |λ|^d ∑_{n=0}^∞ \combin{d-1+n}n\,\left|\G f(n,λ)\right|^2\de λ
\end{equation*}
and polarisation leads to the inner product identity
\begin{equation*}\label{inpr biradiell}
  \inpr{f,g}_{L^2(ℍ^d)} = \frac1{(2π)^{d+1}} ∫_ℝ |λ|^d ∑_{n=0}^∞ \combin{d-1+n}n\,\overline{\G f(n,λ)}\G g(n,λ)\de λ.
\end{equation*}
We shall refer to these as the biradial Plancherel's theorem and Parseval's formula, respectively.

If we now look back at Lemma \ref{stora polynomderivatshs}, we see that for biradial $f$, we could apply the biradial Parseval's formula above if we knew the effect of position space multipliers on $\G$. The next lemma, which is the cornerstone of the present analysis, shows how we can translate such multipliers into derivatives on the biradial frequency side, when analysing the convolution kernel of $φ(-∆)$. It is inspired by the work of S.~Thangavelu \cite{Thangavelu-Multiplier} and that of Bahouri, P.~Gérard and C.-J.~Xu \cite{BahouriGX}.

\begin{lemma}\label{nyckelhs}
  For $m ∈ ℕ$ and $φ ∈ \Continuous^{2m}([0,∞))$ with $φ(-∆)δ ∈ L^2(ℍ^d)$, we have
  \[
    \big|\G [(i\ud s - \sfrac14|\ud z|^2)^mφ(-∆)δ](n,λ)\big| ≤ C_m ∑_{ℓ\, =\, 0}^m\, \sup_{|ω|≤2m}\big|(nλ)^ℓ∂^{m+ℓ}φ\big(|λ|(d + 2n + ω)\big)\big|.
  \]
\end{lemma}
  With $\Continuous^{2m}([0,∞))$ we mean that the continuous derivatives exist in $(0,∞)$ with a finite limit at $0$. One can then extend the domain of $φ$ to a neighbourhood of zero using a Taylor expansion, and multiply with a smooth cut-off function so that the extended function vanishes for arguments smaller than some negative number. Thus we can extend $φ$ to a function in $\Continuous^{2m}(ℝ)$. We shall assume that this is done. For the purposes of the rest of this article, we shall always have $∂^mφ(0) = 0$ for all derivatives up to arbitrarily large $m$, and we may then simply extend $φ$ to vanish for negative arguments.
\begin{proof}
  We write $f = φ(-∆)δ$, which we by density can assume to be Schwartz. It is clear that the mapping $F \mapsto \G [(i\ud s - \sfrac14|\ud z|^2)\G^{-1}F]$ takes $\G(\Schwartz(ℍ^d))$ into $\G(\Schwartz(ℍ^d))$, so we can in particular perform this operation as many times as we like.

  Take for the moment $λ > 0$. Then
\begin{align*}
  \G [(i\ud s &- \sfrac14|\ud z|^2)f](n,λ) = ∫_{ℍ^d}(is - \sfrac14|z|^2)\overline{ℓ_n^λ(v)} f(v)\de v\\
                    &= \combin{d-1+n}n^{-1}∫_{ℍ^d}∂_λ(e^{λ(is -|z|^2/4)})L^{d-1}_n(λ|z|^2/2)f(v)\de v\\
                           \hspace{1em}&\hspace{-1em}= ∂_λ\G f(n,λ) 
                   -  \sfrac12|z|^2\,\combin{d-1+n}n^{-1}∫_{ℍ^d}e^{iλ(s-|z|^2/4)}∂L^{d-1}_n(λ|z|^2/2)f(v)\de v.
\end{align*}
Since the Laguerre polynomial of degree $0$ is constant, the last term above vanishes when $n = 0$. For larger degrees, we use that the Laguerre polynomials satisfy
\[
  t∂L_n^{d-1}(t) = n L_n^{d-1}(t) - (d - 1 + n)L_{n-1}^{d-1}(t)
\]
and therefore
\[
  \G \big[(i\ud s - \sfrac14|\ud z|^2)f\big](n,λ) = ∂_λ\G f(n,λ) - \frac nλ \Big(\G f(n,λ) - \G f(n-1,λ)\Big).
\]
When $λ < 0$, one can similarly utilise the identity
\[
  t∂L_n^{d-1}(t) - tL_n^{d-1}(t) = (n + 1)L_{n+1}^{d-1}(t) - (d + n)L_n^{d-1}(t),
\]
and with this deduce
\[
  \G \big[(i\ud s - \sfrac14|\ud z|^2)f\big](n,λ) = ∂_λ\G f(n,λ) - \frac{d + n}λ \Big(\G f(n + 1,λ) - \G f(n,λ)\Big).
\]
We shall in the following show only the case when $λ$ is positive. The reasoning for negative $λ$ is similar.

The next step is to observe that since $\G f(n,λ) = φ(λ(d + 2n))$, we can let $n$ take real values and exploit that
\[
  ∂_λ \G f(n,λ) = \bigg(\frac nλ + \frac d{2λ}\bigg)\,∂_n \G f(n,λ).
\]
From this follows that
\begin{align*}
  \G \big[(i\ud s &- \sfrac14|\ud z|^2)f\big](n,λ) = \frac d{2λ} ∂_n\G f(n,λ) + \frac nλ∫_0^1 (1-θ)∂_n^2\G f(n - θ, λ)\de θ\\
                  &= d\,∂φ(λ(d + 2n)) + 4nλ∫_0^1(1 - θ)∂^2φ\big(λ(d + 2(n-θ))\big)\de θ.
\end{align*}
We would now like to repeat this argument for higher powers of $is - \frac14|z|^2$, and this can be done right away with the first term. For the second term, however, the relation between $∂_λ\G f$ and $∂_n\G f$ as above does no longer hold. For $j,k,ℓ ∈ ℕ$ and $p$ a polynomial, let $g_{j,k,ℓ,p}$ denote the formula
\[
  g_{j,k,ℓ,p}(n,λ) = (nλ)^j∫_{[0,1]^k}p(θ_1,\ldots,θ_k)∂^ℓφ\big(λ(d + 2(n - θ_1 - \cdots - θ_k))\big)\de θ_1\cdots\dd θ_k.
\]
Here it is important that $φ$ is defined also for negative arguments. This expression satisfies
\[
  ∂_λg_{j,k,ℓ,p}(n,λ) = \frac nλ∂_ng_{j,k,ℓ,p}(n,λ) + g_{j,k,ℓ+1,q}(n,λ)
\]
with $q(θ_1,\ldots,θ_k) = (d - 2θ_1 - \cdots - 2θ_k)p(θ_1, \ldots,θ_k)$. In the case $k = 0$, we simply have $q = d$.
The same reasoning as above then yields that
\begin{align*}
  \G\big[(i\ud s - \sfrac14|\ud z|^2)&\G^{-1}g_{j,k,ℓ,p}\big](n,λ) =\\
                                     &g_{j,k,ℓ+1,q}(n,λ) + \frac nλ∫_0^1(1-θ_{k+1})∂_n^2g_{j,k,ℓ,p}(n-θ_{k+1},λ)\de θ_{k+1}\\
                                     &= g_{j,k,ℓ+1,q}(n,λ) + j(j-1) g_{j-1,k+1,ℓ,π}(n,λ) \\[.4em]
  &\quad+ 4jg_{j,k+1,ℓ+1,π}(n,λ) + 4g_{j+1,k+1,ℓ+2,π}(n,λ),
\end{align*}
where $π$ denotes the polynomial in $m+1$ variables given by $π(θ_1,\ldots,θ_{m+1}) = (1-θ_{m+1})p(θ_1,\ldots,θ_m)$.  

We thus find four terms of the same general form, and the process can be repeated. Notice that at each step we either reduce the power of $nλ$ (if this is nonzero), add a derivative to $φ$, or increase the power of $nλ$ and add two derivatives to $φ$. Combining the contribution from the three types of terms and applying the mean value theorem, we arrive at the requested estimate.
\end{proof} 

We follow up this result with a minor extension. Since $∆$ and $S$ commute and $\F(Sf)(λ) = -iλ\F f(λ)$, one can define the operators $φ(-∆,S)$ by extending Definition \ref{def multiplikator} in the obvious way. These operators clearly have biradial kernels, and minor adjustments to the proof above provides the following estimate.
\begin{corollary}\label{utoekad nyckelhs}
  For $m ∈ ℕ$ and $φ ∈ \Continuous^{2m}([0,∞)×ℝ\setminus\{0\})$ with $φ(-∆, S)δ ∈ L^2(ℍ^d)$, we have
  \begin{align*}
    &\big|\G [(i\ud s - \sfrac14|\ud z|^2)^mφ(-∆,S)δ](n,λ)\big| ≤ \\
    &\qquad C_m ∑_{ℓ\, =\, 0}^m\, \sup_{|ω|≤2m}\big|(nλ)^ℓ∑_{|α|=m+ℓ}∂^αφ\big(|λ|(d + 2n + ω), -iλ\big)\big|.
  \end{align*}
\end{corollary}

We end the section by using Lemma \ref{nyckelhs} to show a crucial property of sub-Laplacian multipliers. A weaker variant of this result follows from the theorem of A.~Hulanicki \cite{Hulanicki}; if $φ$ is Schwartz on $[0, ∞)$, then $φ(-∆)$ maps $\Schwartz(ℍ^d)$ into $\Schwartz(ℍ^d)$.
\begin{theorem}\label{maottlig multiplikator-s}
  Let $φ ∈ C^∞([0,∞))$ have derivatives of at most polynomial growth at all orders. Then $φ(-∆)$ maps $\Schwartz(ℍ^d)$ into $\Schwartz(ℍ^d)$.
\end{theorem}
\begin{proof}
  Let $f$ be Schwartz. We shall show that for any polynomial $p$ over $ℍ^d$ and multi-index $I$, $\sup_v |p(v)\tilde D^Iφ(-∆)f(v)| < ∞$. Since any left invariant vector field can be expressed as a sum $D_i = ∑_jp_j\tilde D^{I_j}$, see \cite{FollandS}, this requirement is equivalent to the one with left-invariant vector fields.

  Let then $M$ be an integer. Lemma \ref{kaerna aer L2} ensures that $φ(-∆)(1-∆)^{-M}δ$ is an $L^2$ function for sufficiently large $M$. We shall use that for any $v,w ∈ ℍ^d$, 
  \[
    p(v) = ∑_j q_j(v \go w^{-1})π_j(w)
  \]
  for some polynomials $q_j$, $π_j$, which follows from expanding $p((v \go w^{-1}) \go w)$. Therefore, by the Cauchy–Schwarz inequality and for some natural numbers $m_j$,
  \begin{align*}
    \big|p(v)\tilde D^Iφ(-∆)f(v)\big| &= \big|p(v)\,(1-∆)^M\tilde D^If \star φ(-∆)(1 - ∆)^{-M}δ(v)\big|\\
                                      &= \Big|∑_j q_j\,(1-∆)^M\tilde D^If \star π_j\,φ(-∆)(1 - ∆)^{-M}δ(v)\Big|\\
                                      \hspace{1.2em}&\hspace{-1.2em}≤ ∑_j \Hnorm{q_j\,(1-∆)^M\tilde D^If} \Hnorm{π_j\,φ(-∆)(1 - ∆)^{-M}δ}\\
                                      &\lesssim ∑_j \Hnorm{(1 + π_j^2)\,φ(-∆)(1 - ∆)^{-M}δ}\\
                                      &\lesssim ∑_j\Hnorm{|\argd|^{2m_j}\,φ(-∆)(1 - ∆)^{-M}δ}\\
                                      &\lesssim ∑_j\Hnorm{(i\ud s - \sfrac14|\ud z|^2)^{m_j}\,φ(-∆)(1 - ∆)^{-M}δ},
  \end{align*}
  where the range of $j$ is changed in the penultimate estimate.
  The \hyperlink{biradial}{biradial Plancherel theorem} then shows that the square of term $j$ in the last row above equals
  \[
    \frac1{(2π)^{d+1}}∫_ℝ|λ|^d∑_{n=0}^∞\combin{d-1+n}n \big|\G[(i\ud s - \sfrac14|\ud z|^2)^{m_j}\, φ(-∆)(1 - ∆)^{-M}δ](n,λ)\big|^2\de λ,
  \]
  and with the help of Lemma \ref{nyckelhs}, we can estimate this from above with
  \[
    ∫_0^∞λ^d ∑_{n=0}^∞ n^{d-1}(1 + λ(1 + n))^{2k - 2M}\de λ,
  \]
  for some integer $k = k(m_i,ℓ,φ)$. This is finite for sufficently large $M$.
\end{proof}

\section{Hardy spaces}\label{Hardyrumsavdelning}
We shall follow the atomic characterisation of Hardy spaces given by Folland and Stein \cite{FollandS}. For $p ∈ (1,∞)$, we define the Hardy space of power $p$, denoted $H^p(ℍ^d)$, to be $L^p(ℍ^d)$. For $p ∈ (0, 1]$ we let $\mathcal N = \lfloor Q(1/p - 1)\rfloor$.
Define a $p$-atom on $ℍ^d$ to be a compactly supported function $a ∈ L^∞(ℍ^d)$ such that
\begin{enumerate}
\item There is a closed ball $B ⊂ ℍ^d$ with $B \supset \supp a$ and $\|a\|_{L^∞(ℍ^d)} ≤ |B|^{ - 1/p}$.
\item For any polynomial $q$ of homogeneous degree no larger than $\mathcal N$, we have \rule{0pt}{1.2em}$∫_{ℍ^d}a(v)q(v)\de v = 0$.
\end{enumerate}
Then $H^p(ℍ^d)$ consists of the tempered distributions of the form $∑_{i=0}^∞ c_ia_i$, converging in $\Schwartz'(ℍ^d)$, with each $a_i$ a $p$-atom and where the $c_i ∈ ℂ$ satisfy $∑_{i=0}^∞|c_i|^p < ∞$. In other words, this definition of Hardy spaces mimics that on Euclidean space.  We define the $H^p$ quasinorm by $\|f\|_{H^p(ℍ^d)} = \inf\,(∑_{i=0}^∞|c_i|^p)^{1/p}$ where the infimum is over all atomic decompositions $f = ∑_{i=0}^∞c_ia_i$.

Just as over Euclidean space, the Hardy spaces over $ℍ^d$ can be expressed with several equivalent norms. Two more that shall be of interest to us are constructed using the heat flow $e^{t∆}$.

We shall have use of the Lusin norm
\[
  \|f\|_{H^p_\mathrm{L}(ℍ^d)} = \Bigg\|\Bigg(∫_0^∞∫_{|w^{-1}\,\go\, \ud v| < t}|(-t∆)e^{t∆}f(w)|^2\,t^{-Q-1}\de w\de t\Bigg)^{1/2}\Bigg\|_{L^p(ℍ^d)},
\]
as well as the Littlewood--Paley norms, for positive integers $N$,
\begin{equation}\label{LP-norm}
  \begin{aligned}
  \|f\|_{H^p_N(ℍ^d)} &= \Bigg\|\Bigg(∫_ℝ|(-t∆)^Ne^{t∆}f|^2\,\frac1t\de t\Bigg)^{1/2}\Bigg\|_{L^p(ℍ^d)}\\
                     &= \Bigg\|\Bigg(∫_ℝ|ψ(-t∆)f|^2\,\frac1t\de t\Bigg)^{1/2}\Bigg\|_{L^p(ℍ^d)},
\end{aligned}
\end{equation}
where we have defined
\[
  ψ : σ \mapsto σ^Ne^{-σ}.
\]
As elaborated in Chapter {\capfigures7} of \cite{FollandS}, each of these norms is equivalent with the atomic norm of $H^p(ℍ^d)$ described above, and we shall in the sequel write $H^p_N(ℍ^d) = H^p(ℍ^d)$. It will be paramount that we may take this $N$ to be as large as need be.

The dual of $H^1(ℍ^d)$ is the space $\BMO(ℍ^d)$ of functions of bounded mean oscillation, modulo constant functions. BMO consists of the locally integrable functions $f$ on $ℍ^d$ such that the seminorm
\[
  \|f\|_{\BMO(ℍ^d)} = \sup_B\frac1B∫_B\bigg|\,f(v) - \frac1B∫_Bf(w)\de w\bigg|\de v < ∞,
\]
where the supremum is over all balls $B$ in $ℍ^d$.

We end this section with a lemma about how to effectively commute left-invariant vector fields $D^I$ with the frequency concentrating operator of the Littlewood--Paley norm above.  When making use of this lemma, we shall always ignore the second sum with $Φ'$.

\begin{lemma}\label{utbytarhs}
  For any homogeneous polynomial $p$ of degree $m$ and for any multi-index $I$, there are homogeneous polynomials $q_n$ of degree $n = 0, \ldots, m$ and homogeneous polynomials $P, P'$ on $[0,∞) × ℂ$ of degree $N$ such that for all Schwartz $f$ and $r > 0$, we have, with $N$ and $ψ$ as in equation \pref{LP-norm},
  \begin{align*}
    \hspace{3em}&\hspace{-3em}\big\|p ψ(r^{-2}(-∆))D^If\big\|_{L^2(ℍ^d)} ≤\\
    &\ANDRA C_{m,I}∑_{|J| = |I|} ∑_{n=0}^{m}r^{n-m}\big\|q_nD^JΦ(r^{-2}(-∆), r^{-2}S)f\big\|_{L^2(ℍ^d)}\\
    &+C_{m,I}∑_{|J| = |I|} ∑_{n=0}^{m}r^{n-m}\big\|q_nD^JΦ'(r^{-2}(-∆), r^{-2}S)f\big\|_{L^2(ℍ^d)},
  \end{align*}
  where
  \[
    Φ(a,b) = (1 + a^{2N})^{-1}P(a,b),\quad Φ'(a,b) = (1 + a^{2N})^{-1}P'(a,b).
  \]
\end{lemma}
\begin{proof}
  If we define the $2d×2d$ matrix $J$ to act on the $2d$-vector $E = (X,Y)$ by
  \[
    JE = (Y,-X),
  \]
  we have from the \hyperlink{commutation relation} {commutation relation of the left-invariant vector fields} (page \pageref{commutation relation}) that
  \(
    ∆E = E∆ + 2JES,
  \)
  and an induction argument shows that
  \[
    ∆^nE = ∑_{k=0}^n\combin nk(2J)^kE\,∆^{n-k}S^k, \qquad E∆^n = ∑_{k=0}^n\combin nk∆^{n-k}S^k(-2J)^kE.
  \]
  Therefore, for any multi-indices $I,K$ and positive integers $j$, there are constants $C_{K,j}$, $C_{K,j}'$ such that
  \begin{equation}\label{kommutera Laplace och D}
    \begin{aligned}
      (-∆)^nD^I &= ∑_{|K| = |I|}∑_{j=0}^nC_{K,j}D^K(-∆)^{n-j}S^j\\
      D^I(-∆)^n &= ∑_{|K| = |I|}∑_{j=0}^nC_{K,j}'(-∆)^{n-j}S^jD^K.
    \end{aligned}
  \end{equation}
  Using these identities, we find that 
  \begin{align*}
    &ψ(r^{-2}(-∆))D^I =\\[.4em]
    &\ANDRA e^{r^{-2}∆}(-r^{-2}∆)^N D^I\big(1 + (-r^{-2}∆)^{2N}\big)\big(1 + (-r^{-2}∆)^{2N}\big)^{-1} \\
                     &= e^{r^{-2}∆}\Bigg(∑_{|K| = |I|}∑_{j=0}^NC_{K,j}D^Kr^{-2N}(-∆)^{N-j}S^j \\
                     &\quad + \!∑_{|K| = |I|}∑_{j=0}^NC_{K,j}'r^{-4N}(-∆)^{2N-j}S^jD^K(-r^{-2}∆)^N\Bigg)\big(1 + (-r^{-2}∆)^{2N}\big)^{-1}\\
                     &=: e^{r^{-2}∆}∑_{|K| = |I|}(D^KA'_K + B_KD^KB')(1 + (-r^{-2}∆)^{2N})^{-1},
  \end{align*}
  where we have defined
  \begin{equation*}\label{A'BB'}
    \begin{aligned}
      A'_K &= ∑_{j=0}^NC_{K,j}r^{-2N}(-∆)^{N-j}S^j\\
      B_K &= ∑_{j=0}^NC_{K,j}'r^{-4N}(-∆)^{2N-j}S^j,\qquad  B' = (-r^{-2}∆)^N.
    \end{aligned}
  \end{equation*}
  Now, we can assume that $p$ is monomial, and there are then, just as in the proof of Theorem \ref{maottlig multiplikator-s}, homogeneous polynomials $q_n,π_{m-n}$ of degrees $n$ and $m-n$, respectively, such that, for any $v,w ∈ ℍ^d$, $p(v) =∑_{n=0}^{m} q_n(v \go w^{-1})π_{m-n}(w)$. Therefore, using Young's inequality,
  \begin{align*}
    &\big\|pψ(r^{-2}(-∆))D^If\big\|_{L^2(ℍ^d)} ≤\\[.4em]
    & ∑_{|K| = |I|}\big\|pe^{r^{-2}∆}(D^KA'_K + B_KD^KB')(1 + (-r^{-2}∆)^{2N})^{-1}f\big\|_{L^2(ℍ^d)}\\
    &=∑_{|K| = |I|}\Big\|p\big[D^KA'_K(1 + (-r^{-2}∆)^{2N})^{-1}f \star e^{r^{-2}∆}δ \\
    &\hspace{5em} + D^KB'(1 - (-r^{-2}∆)^{2N})^{-1}f \star e^{r^{-2}∆}B_Kδ\big] \Big\|_{L^2(ℍ^d)}\\
    &=∑_{|K| = |I|}∑_{n=0}^m\big\|q_nD^KA'_K(1 + (-r^{-2}∆)^{2N})^{-1}f \star π_{m-n}e^{r^{-2}∆}δ \\
    &\hspace{5em} + q_nD^KB'(1 - (-r^{-2}∆)^{2N})^{-1}f \star π_{m-n}e^{r^{-2}∆}B_Kδ \big\|_{L^2(ℍ^d)}\\
    &≤ ∑_{|K| = |I|}∑_{n=0}^{m}\bigg(\big\|r^{n-m}q_nD^KA'_K(1 + (-r^{-2}∆)^{2N})^{-1}f\big\|_{L^2(ℍ^d)}\\[-.2em]
    &\hspace{18em}×\ \big\|r^{m-n} π_{m-n}e^{r^{-2}∆}δ\big\|_{L^1(ℍ^d)} \\
    &+ \big\|r^{n-m}q_nD^KB'(1 + (-r^{-2}∆)^{2N})^{-1}f\big\|_{L^2(ℍ^d)}\big\|r^{m-n}π_{m-n}e^{r^{-2}∆}B_Kδ \big\|_{L^1(ℍ^d)}\bigg).
  \end{align*}
  We can then use homogeneity with identity \pref{multiplikatorutdragning2}, a large integer $M$, Hölder's inequality, the \hyperlink{biradial}{biradial Plancherel's theorem} and Corollary \ref{utoekad nyckelhs} to deduce that
  \begin{align*}
    &\big\|r^{m-n} π_{m-n}e^{r^{-2}∆}B_Kδ\big\|_{L^1(ℍ^d)} = \Bigg\|∑_{j=0}^NC'_{K,j}π_{m-n}\,e^∆(-∆)^{N-j}S^jδ\Bigg\|_{L^1(ℍ^d)}\\
    &\lesssim ∑_{j=0}^N\big\|(1 + |·|^2)^{-M}\big\|_{L^2(ℍ^d)}\big\|(1 + |·|^2)^Mq_{m-n}\,e^∆(-∆)^{N-j}S^jδ\big\|_{L^2(ℍ^d)}\\
          &\lesssim ∑_{j=0}^N \big\|(1 + |·|^2)^M(1 + |·|^{2(m-n)})\,e^∆(-∆)^{N-j}S^jδ\big\|_{L^2(ℍ^d)} \lesssim 1,
  \end{align*}
  and similarly for the $L^1$-norm in the first of the two terms above.
\end{proof}

\section[Proof of Theorem 1.3]{Proof of Theorem \ref{allmaenna s}}\label{allmaenna ss avdelning}
Let us abbreviate the operator of interest $T = φ_t(-∆)$. We shall also denote
\[
  μ =  -γQ(1/p - 1/2) - |I|/2 - |I'|/2,
\]
and write $\jap t = 1 + t$, so that the requirement on $φ_t$ reads that for some $γ ≥ 0$,
\[
  ∂^mφ_t(σ) \lesssim [\jap tσ^{γ-1}]^m(1 + σ)^μ.
\]
We shall start by showing the theorem for $p ∈ (0,1)$ in the special case when $I'$ is empty, and then show this  for $p = 1$. Afterwards we make the extension to general $I'$, and we can deduce the $p = ∞$ case from duality. Finally, we show how the cases with $p ∈ (1, ∞)$ follow from interpolation.

\subsection{Boundedness on quasi-Banach spaces}
The proof of boundedness on $H^p(ℍ^d)$ with $p < 1$ and $|I'| = 0$ consists of three steps.

\subsubsection*{Step 1: Atomic analysis and frequency decomposition}
Using the atomic decomposition of Hardy spaces, it will for $p ≤ 1$ suffice to show that
\begin{equation}\label{verkan på atom}
   \|Ta\|_{H^p(ℍ^d)} \lesssim \jap t^{Q(1/p - 1/2)}\|a\|_{H^p(ℍ^d)} = \jap t^{Q(1/p - 1/2)} =: t_p
\end{equation}
for a fixed arbitrary atom $a$ supported in some ball $B$, in order to conclude that $\|Tf\|_{L^p(ℍ^d)} \lesssim \|f\|_{H^p(ℍ^d)}$ holds for arbitrary elements of $H^p(ℍ^d)$. Since the norms are translation invariant, we can assume that $B$ is centred at the origin. 

We shall use one of the Littlewood--Paley norms from the previous section, with a large positive integer $N$. It will turn out to be natural to use the particular one-parameter set of operators $ψ(r^{-2}(-∆)) = (-r^{-2}∆)^Ne^{r^{-2}∆}$. These operators hence concentrate frequency around $r$. If we further abbreviate the operators
\[
  T_r = ψ(r^{-2}(-∆))D^IT,
\]
we can write
\begin{align*}
  \hspace{1em}&\hspace{-1em}\|Ta\|_{H^p(ℍ^d)} = \\
  &\Bigg\|\Bigg(∫_ℝ|ψ(r^{-2}(-∆))D^ITa|^2\db r\Bigg)^{1/2}\Bigg\|_{L^p(ℍ^d)} = \Bigg\|\Bigg(∫_ℝ|T_ra|^2\db r\Bigg)^{1/2}\Bigg\|_{L^p(ℍ^d)}.
\end{align*}
Before we start decomposing this expression, we notice that the $D^I$ operator is in between the sub-Laplacian multipliers $ψ(r^{-2}(-∆))$ and $T$, which is not the form we want if we subsequently want to apply Lemma \ref{stora polynomderivatshs}. To come around this issue is the purpose of Lemma \ref{utbytarhs}, which shall come into play shortly.

Now, we start by subdividing the norm into two regions; one inner, localised around $B$, and the other being the outside of this region. Write $2^ℓ$ for the radius of $B$, and let $B^*$ be the ball cocentric with $B$ of radius $2·2^ℓ\jap t$. We then decompose
\[
  \|Ta\|_{H^p(ℍ^d)} ≤ \|Ta\|_{H^p(B^*)} + \|Ta\|_{H^p(B^{*c})}.
\]
For the inner part, we use Hölder's inequality to find that
\begin{align}
  \|Ta\|_{H^p(B^*)} &= \Bigg\|\Bigg(∫_ℝ|T_ra|^2\db r\Bigg)^{1/2}\Bigg\|_{L^p(B^*)} \label{foersta T_ra-normen}\\
                      &≤ \|1\|_{L^{2p/(2 - p)}(B^*)}\Bigg\|\Bigg(∫_ℝ|T_ra|^2\db r\Bigg)^{1/2}\Bigg\|_{L^2(ℍ^d)}\notag \\
                    &\lesssim (2^ℓ\jap t)^{Q(1/p-1/2)}\Bigg(∫_ℝ\|T_ra\|_{L^2(ℍ^d)}^2\db r\Bigg)^{1/2}\notag
\end{align}
Lemma \ref{utbytarhs} with $m = 0$ shows that this is estimated from above by
\[
  t_p|B|^{1/p-1/2}∑_{|J|=|I|}\Bigg(∫_ℝ\|D^JΦ(r^{-2}(-∆),r^{-2}S)Ta\|_{L^2(ℍ^d)}^2\db r\Bigg)^{1/2},
\]
and Lemma \ref{derivatuskshs} in turn gives the upper estimate
\[
  t_p|B|^{1/p-1/2}\Bigg(∫_ℝ\|(-∆)^{|I|/2}Φ(r^{-2}(-∆),r^{-2}S)Ta\|_{L^2(ℍ^d)}^2\db r\Bigg)^{1/2}.
\]
We then define $Ψ ∈ \Smooth([0,∞))$ from $Ψ(σ) = σ^N/(1 + σ^{2N})$ and make use of Plancherel's theorem to see that this expression equals
\begin{align*}
                 &t_p|B|^{1/p-1/2}\Bigg(∫_ℝ|λ|^d∑_{m∈ℕ^d}∫_ℝ|Φ(r^{-2}|λ|(d + 2|m|), r^{-2}iλ)|^2\db r\notag\\
  &\TERM ×\ \big(|λ|(d + 2|m|)\big)^{|I|}|φ_t(|λ|(d + 2|m|)|^2 \|\F a(λ)η_m^λ\|_{L^2(ℝ^d)}^2\de λ\Bigg)^{1/2}\notag\\ 
                  &\lesssim t_p|B|^{1/p-1/2}\Bigg(∫_ℝ|λ|^d∑_{m∈ℕ^d}∫_ℝ|Ψ(r^{-2}|λ|(d + 2|m|))|^2\db r\notag\\
                 &\TERM ×\ \big(|λ|(d + 2|m|)\big)^{|I|}|φ_t(|λ|(d + 2|m|)|^2 \|\F a(λ)η_m^λ\|_{L^2(ℝ^d)}^2\de λ\Bigg)^{1/2}\notag\\
                 &\lesssim t_p|B|^{1/p-1/2}\Bigg(∫_ℝ|λ|^d∑_{m∈ℕ^d} \|\F a(λ)η_m^λ\|_{L^2(ℝ^d)}^2\de λ\Bigg)^{1/2}\\
                   &= t_p|B|^{1/p-1/2}\|a\|_{L^2(ℍ^2)}.
\end{align*}
Finally, the first atomic property shows that this quantity is bounded from above by $t_p$.

To continue with the part outside of $B^*$, we claim that estimate \pref{verkan på atom} is satisfied if the following requirement hold: There are positive constants $ε$, $ε'$ such that we have for each integer $k$ that
\begin{align}
  \label{krav1}
  \Bigg\|\Bigg(∫_{2^k}^{2^{k+1}}|T_ra|^2\db r\Bigg)^{1/2}\Bigg\|_{L^p(B^{*c})} &\lesssim 2^{-ε(k+ℓ)}t_p \qquad\text{when $k ≥ -\floor ℓ$}\\
    \label{krav2}
  \Bigg\|\Bigg(∫_{2^k}^{2^{k+1}}|T_ra|^2\db r\Bigg)^{1/2}\Bigg\|_{L^p(B^{*c})} &\lesssim 2^{ε'(k+ℓ)}t_p \qquad\text{when $k < -\floor ℓ$}.
\end{align}
In other words, we would require rapid decay for the high respectively low frequency portions of the operator, relative to the inverted size of $B$. These two demands suffice, since if they hold, 
\begin{align*}
  \|Ta\|_{H^p(B^{*c})}^p &≤ ∑_{k∈ℤ}\,\Bigg\|\Bigg(∫_{2^k}^{2^{k+1}}|T_ra|^2\db r\Bigg)^{1/2}\Bigg\|_{L^p(B^{*c})}^p\\
  &\lesssim ∑_{k=-∞}^{-\floor ℓ-1} 2^{-εp(k+ℓ)}t_p^p + ∑_{k=-\floor ℓ}^∞ 2^{ε'p(k+ℓ)}t_p^p \lesssim t_p^p.
\end{align*}

\subsubsection*{Step 2: Translation to kernel estimates}
  The second step in the proof is to show how estimates \pref{krav1} and \pref{krav2} follow from certain $L^2$ estimates on the kernels of $T_r$.

In the upcoming analysis, we will find that the position space multiplier $|·|^2$ acting on $T_r$ results in an asymptotic growth $\jap t^2 2^{-2ck}$, where $c$ is a real constant that is different in two regions of frequency-time. To separate these regions, we take a nonnegative $χ^\mathup L ∈ \Smooth_c([0,∞))$ that equals $1$ close to $0$, set $χ^\mathup H = 1 - χ^\mathup L$ and thus form the low and high frequency-time operators
\[
  T_r = T^\mathup L_r + T^\mathup H_r := T_rχ^\mathup L\big(\jap t(-∆)^γ\big) + T_rχ^\mathup H\big(\jap t(-∆)^γ\big).
\]
However, since the treatment of both series of operators will be the same for the initial part of the analysis, we shall suppress this difference for now, and return to it at the stage when we fix $c$.

 To show the high frequency decay requirement \pref{krav1}, we shall use the expected asymptotic growth mentioned above. For any non-negative integer $M$, we may multiply and divide by a polynomial tailored from this property and use Hölder's inequality, to find that
\begin{align*}
  \hspace{1em}&\hspace{-1em}\Bigg\|\Bigg(∫_{2^k}^{2^{k+1}}|T_ra|^2\db r\Bigg)^{1/2}\Bigg\|_{L^p(B^{*c})} ≤\\
  & \big\|(1 + \jap t^{-2}2^{2ck}|·|^2)^{-M}\big\|_{L^{2p/(2-p)}(B^{*c})}\\
  &×\ \Bigg(∫_{2^k}^{2^{k+1}}\big\|(1 + \jap t^{-2}2^{2ck}|·|^2)^M T_ra\big\|_{L^2(B^{*c})}^2\db r\Bigg)^{1/2}\\[.4em]
 &\lesssim (\jap t2^{-ck})^{Q(1/p-1/2)} \big\|(1 + \jap t^{-2}2^{2ck}|·|^2)^M T_{\ud r}a\big\|_{L^2(B^{*c}×I_k)},
\end{align*}
where we write $I_k$ for the interval $[2^k, 2^{k+1}]$ equipped with the measure $\dd r/r$. Since $M$ is arbitrary, we may equivalently study the expression
\begin{align}
  &(\jap t2^{-ck})^{Q(1/p-1/2)}(\jap t^{-2}2^{2ck})^M \big\||·|^{2M} T_{\ud r}a\big\|_{L^2(B^{*c}×I_k)}\label{foersta T^J}\\
  &\lesssim t_p2^{-ckQ(1/p-1/2)}(\jap t^{-2}2^{2ck})^M \notag\\
  &\hspace{3em}×\ ∑_{|J| = |I|} ∑_{m=0}^{2M}2^{(m-2M)k}\big\|p_mD^JΦ(\ud r^{-2}(-∆), \ud r^{-2}S)Ta\big\|_{L^2(B^{*c}×I_k)}\notag\\
  &=: ∑_{|J| = |I|} ∑_{m=0}^{2M}t_p\jap t^{-2M}2^{k[-cQ(1/p-1/2)+2(c-1)M + m]}  \big\|p_mT^J_{\ud r}a\big\|_{L^2(B^{*c}×I_k)},\notag
\end{align}
where we have used Lemma \ref{utbytarhs} and defined
\[
  T^J_r = D^JΦ(r^{-2}(-∆), r^{-2}S)T.
\]
Write then $K^J_r = δ \circ T^J_r \circ P$ for the right convolution kernel of $T^J_r$. The summand of the expression above then equals
\begin{equation}
  \hspace{-1em}t_p\jap t^{-2M}2^{k[-cQ(1/p-1/2) + (c-1)2M + m]}  \Bigg\|p_m(\ud v)∫_BK^J_{\ud r}(w^{-1}\go\ud v)a(w)\de w\Bigg\|_{L^2(B^{*c}×I_k)}\hspace{-3em}\label{problem foer p = 1}
\end{equation}
and we may apply Minkowski's inequality to deduce that it is bounded from above by
\begin{equation}\label{dags foer polynomtrix}
  t_p\jap t^{-2M}2^{k[-cQ(1/p-1/2) + (c-1)2M + m]}  ∫_B|a(w)|\Big\|p_m(\ud v)K^J_{\ud r}(w^{-1}\go\ud v)\Big\|_{L^2(B^{*c}×I_k)}\de w.
\end{equation}
Now, for $v ∈ B^{*c}$ and $w ∈ B$, we have for some homogeneous polynomials $q_n$ and $π_{m-n}$ of degrees $n$ and $m-n$ that
\begin{align*}
  p_m(v) &= p_m(w \go (w^{-1}\go v)) = ∑_{n=0}^mq_n(w)π_{m-n}(w^{-1}\go v) \\
         &\lesssim ∑_{n=0}^m |q_n(w^{-1}\go v)π_{m-n}(w^{-1}\go v)|.
\end{align*}
For an upper estimate of the expression above, we may therefore make the substitution $p_m(v) → ∑_n|q_n(w^{-1}\go v)π_{m-n}(w^{-1}\go v)|$, but since each product $q_nπ_{m-n}$ is a homogeneous polynomial of degree $m$, we shall for brevity suppress this sum and write again $p_m(w^{-1}\go v)$ for any of these terms. With this slight abuse of notation, expression \pref{dags foer polynomtrix} is estimated from above by
\begin{align*}
  &t_p\jap t^{-2M}2^{k[-cQ(1/p-1/2) + (c-1)2M + m]} \\
  &\hspace{3em}×\ ∫_B|a(w)|\Big\|p_m(w^{-1}\go \ud v)K^J_{\ud r}(w^{-1}\go\ud v)\Big\|_{L^2(B^{*c}×I_k)}\de w\\
    &≤ t_p\jap t^{-2M}2^{k[-cQ(1/p-1/2) + (c-1)2M + m]}  ∫_B|a(w)|\Big\|p_m(\ud v)K^J_{\ud r}(\ud v)\Big\|_{L^2(H^d×I_k)}\de w.
\end{align*}
Finally, the first atomic property shows that this expression is bounded from above by 
\[
  t_p\jap t^{-2M}2^{k[-cQ(1/p-1/2) + (c-1)2M + m] - ℓQ(1/p-1)}  \big\|p_m K^J_{\ud r}\big\|_{L^2(ℍ^d×I_k)}.
\]
This should be compared to the right-hand side of the requested estimate \pref{krav1}. For the $ℓ$-dependence to match, we see that we can take $ε = Q(1/p - 1)$, which is positive for $p < 1$. In order for all other dependencies to cancel out, we require that, for all nonnegative $m ≤ 2M$,
\begin{equation}\label{stegtillkrav1}
  \big\|p_mK^J_{\ud r}\big\|_{L^2(ℍ^d×I_k)} ≤ C_M \jap t^{2M}2^{k[Q/2 + (c-1)Q(1/p-1/2)  - c2M + 2M - m]}.
\end{equation}

The next task is to derive a similar estimate that provides for the low frequency decay requirement \pref{krav2}. We start by making use of the second atomic property in order to subtract the Taylor polynomials of Theorem \ref{Taylors s}:
\begin{align*}
  T_r^Ja(v) &=  ∫_{ℍ^d}K^J_r(w^{-1} \go v)a(w)\de w \\
  &= ∫_{ℍ^d} \Big(K^J_r(w^{-1} \go v) - ∑_{j=0}^{\mathcal N} \frac1{j!} \tilde M_{w^{-1}}^jK^J_r(v)\Big)\,a(w)\de w.
\end{align*}
It then follows from Theorem \ref{Taylors s} and its proof that
\begin{align*}
  &K^J_r(w^{-1}\go v) - ∑_{j=0}^{\mathcal N} \frac1{j!} \tilde M_{w^{-1}}^jK^J_r(v) = \frac1{\mathcal N!}∫_0^1(1-θ)^{\mathcal N}_{\phantom θ}∂_θ^{\mathcal N+1}K^J_r(δ_θ(w^{-1}) \go v)\de θ\\
  &= ∫_0^1∑_{\substack{a+b≤\mathcal N+1\\b≤a}}\frac{C_{a,b}}{\mathcal N!}(1-θ)^{\mathcal N}_{\phantom θ}θ^{a-b}(\mathring sS)^a(\mathring z·\tilde E)^{\mathcal N+1-a-b}K^J_r(δ_θ(w^{-1}) \go v)\de θ,
\end{align*}
where we have put $(\mathring z, \mathring s) = w^{-1}$. We can therefore replace $K^J_r(w^{-1}\go v)$ in the previous reasoning with the expression above, thus finding that formula \pref{problem foer p = 1} satisfies
\begin{align*}
  &t_p\jap t^{-2M}2^{k[-cQ(1/p-1/2)+2(c-1)M + m]}  \Bigg\|p_m(\ud v)∫_BK^J_{\ud r}(w^{-1}\go\ud v)a(w)\de w\Bigg\|_{L^2(B^{*c}×I_k)}\\
  &=t_p\jap t^{-2M}2^{k[-cQ(1/p-1/2)+2(c-1)M + m]}  \Bigg\|p_m(\ud v)∫_B∫_0^1∑_{\substack{a+b≤\mathcal N+1\\b≤a}}\frac{C_{a,b}}{\mathcal N!}(1-θ)^{\mathcal N}_{\phantom θ}θ^{a-b}\\
  &\quad ×\ (\mathring sS)^a(\mathring z·\tilde E)^{\mathcal N+1-a-b}K^J_r(δ_θ(w^{-1}) \go v)\de θ\,a(w)\de w\Bigg\|_{L^2(B^{*c}×I_k)}\\
  &\lesssim t_p\jap t^{-2M}2^{k[-cQ(1/p-1/2)+2(c-1)M + m]} ∫_B|a(w)| \Bigg\|p_m(\ud v)∫_0^1∑_{\substack{a+b≤\mathcal N+1\\b≤a}}\frac{C_{a,b}}{\mathcal N!}\\
  &\quad ×\ (1-θ)^{\mathcal N}_{\phantom θ}θ^{a-b}(\mathring sS)^a(\mathring z·\tilde E)^{\mathcal N+1-a-b}K^J_r(δ_θ(w^{-1}) \go \ud v)\de θ\Bigg\|_{L^2(B^{*c}×I_k)}\de w.
\end{align*}
To estimate this expression, we would now like to put it in a form to which we could apply Lemma \ref{stora polynomderivatshs}, which in particular requires left-invariant vector fields. This is addressed by the following lemma.
\begin{lemma}
  For all multi-indices $I$, radial $g ∈ \Smooth(ℍ^d)$ and biradial $f ∈ \Continuous^{|I|}(ℍ^d)$,
  \[
    \|g\tilde D^If\|_{L^p(U)} \lesssim \|g D^If\|_{L^p(U)}
  \]
  for all $p ∈ (0,∞]$, and all measurable $U ⊂ ℍ^d$ with $U^{-1} = U$.
\end{lemma}
\begin{proof}
  Firstly, since the biinvariant $S$ commutes with all components of $D$ and $\tilde D$, we can push any $S$ in $\tilde D^I$ up against $f$, forming $S^nf$, which is again a biradial function. It therefore suffices to show the estimate with $D$ replaced with $E$.

  We can write $\tilde E_j = -PE_jP$ for all $j$, with $Pf(v) = f(v^{-1})$ as before. Therefore $\tilde E^I = (-1)^{|I|}PE^IP$. From representation \pref G, we then find that $\overline{Pf} = \G ^{-1}[\overline{\G f}]$, and hence if we write $f = f_\mathrm R + f_\mathrm I := \G ^{-1}[\Re \G f] + \G ^{-1}[i\;\!\Im \G f]$,
\[
  \tilde E^If_\mathrm R = (-1)^{|I|}  \mathrm{Cc}^{|I|}[PE^If_\mathrm R], \quad \tilde E^If_\mathrm I =   \mathrm{Cc}^{|I|}[PE^If_\mathrm I],
\]
where $\mathrm{Cc}$ is the mapping that returns the complex conjugate of its argument.
\end{proof}
With this property set, we may hence replace $\tilde E$ with $E$ in the expression above the lemma. We also translate from the left with $δ_θ(w)$ in the $L^2$ norm and apply Jensen's inequality, and thus find that the $L^2$-norm above can be estimated from above by
\begin{align*}
  & \sqrt{∫_0^1\,\Bigg\|\,p_m(w^{-1} \go δ_θ(w) \go \ud v)∑_{\substack{a+b≤\mathcal N+1\\b≤a}}\big|(\mathring sS)^a(\mathring z·E)^{\mathcal N+1-a-b}K^J_{\ud r}(\ud v)\big|\Bigg\|_{L^2(B^{*c}×I_k)}^2\!\!\dd θ}  \\
  &\quad \lesssim ∑_{\substack{a+b≤\mathcal N+1\\b≤a}} |w|^{\mathcal N+1+a-b}∑_{|K|=\mathcal N+1+a-b}\big\|p_m(\ud v)D^KK^J_{\ud r}(\ud v)\big\|_{L^2(ℍ^d×I_k)}.
\end{align*}
Reinserting this into the reasoning above, we can deduce that formula \pref{problem foer p = 1} satisfies
\begin{align*}
  &\FOERSTA t_p\jap t^{-2M}2^{k[-cQ(1/p-1/2)+2(c-1)M + m]}  \big\|p_mT_{\ud r}^Ja\big\|_{L^2(B^{*c}×I_k)}\\[.6em]
  &\lesssim ∑_{\substack{a+b≤\mathcal N+1\\b≤a}}t_p\jap t^{-2M}2^{k[-cQ(1/p-1/2)+2(c-1)M + m]+ℓ[Q(1/p-1)+\mathcal N+1+a-b]}\\
  &\TERM×\ ∑_{|K|=\mathcal N+1+a-b}\big\|p_mD^KK^J_{\ud r}\big\|_{L^2(ℍ^d×I_k)}
\end{align*}
In order to show estimate \pref{krav2}, this expression should be estimated from above by $t_p2^{ε'(k+ℓ)}$, and we may set for each $a,b$ that $ε' = \mathcal N + 1 + a - b - Q(1/p - 1)$, which is always positive. It hence suffices to show that for any nonnegative integers $m,M$  with $m ≤ 2M$ and any multi-index $K$,
\begin{equation}\label{stegtillkrav2}
  \bigl\|p_mD^KK^J_{\ud r} \big\|_{L^2(ℍ^d×I_k)} \lesssim \jap t^{2M}2^{k[Q/2 + Q(c-1)(1/p - 1/2) + |K| - c2M + 2M - m]}.
\end{equation}
Notice that requirement \pref{stegtillkrav1} is a special case of this one, with the multi-index $K$ being the empty list.

\subsubsection*{Step 3: Proof of the kernel estimates}
We shall now show estimate \pref{stegtillkrav2}, hence also estimate \pref{stegtillkrav1}. Recall that we divided the operator through time and frequency into a high frequency-time part and one with low frequency-time. These were
\begin{align*}
  T^\mathup H &= φ_t(-∆)χ^\mathup H(\jap t(-∆)^γ)\\
  T^\mathup L &= φ_t(-∆)χ^\mathup L(\jap t(-∆)^γ),
\end{align*}
where $χ^\mathup H$ is supported away from $0$ and $χ^\mathup L$ is supported close to $0$. After we have applied Lemma \ref{utbytarhs}, the objects of study are hence the operators
\begin{align*}
  D^KD^JΦ(r^{-2}(-∆), r^{-2}S)φ_t(-∆)χ^\mathup H(\jap t(-∆)^γ)\\
  D^KD^JΦ(r^{-2}(-∆), r^{-2}S)φ_t(-∆)χ^\mathup L(\jap t(-∆)^γ),
\end{align*}
and in particular their convolution kernels, which we denote $D^KK^{\mathup HJ}_r$ and $D^KK^{\mathup LJ}_r$. We use here that the kernels $K_T$ of left-invariant operators $T$ satisfy $D^KTf = (f \star K_T )\star D^Kδ = f \star (K_T \star D^Kδ) = f \star D^KK_T$, so that we can treat $D^K$ and $D^J$ together.

We are going to use Lemma \ref{stora polynomderivatshs} to replace the mixed derivatives $D^KD^J$ with functions of the sub-Laplacian, and then apply the biradial Parceval's formula and Lemma \ref{nyckelhs}, and treat the two regions of frequency-time differently. We begin by computing the derivative estimates that we shall need. For the high frequency-time region, we see that\label{frekvensuppdelning}
  \begin{align}
  \hspace{2em}&\hspace{-2em}∂_σ^ℓ[σ^jφ_t(σ)χ^\mathup H(\jap tσ^γ)]=\label{hf-derivator}\\[.4em]
  &\ANDRA ∑_{\substack{i+k+n+m = ℓ\\i < ℓ,\, k ≤ j}}C_{i,k,n,m}(1 + σ)^μσ^{(γ-1)(n + m) + j - k - i}\jap t^{n+m}∂^mχ^\mathup H(\jap tσ^γ)\notag\\
  &\lesssim ∑_{\substack{i+k+n+m = ℓ\\i < ℓ,\, k ≤ j}}(\jap tσ^γ)^{i+k}(1 + σ)^μσ^{(γ-1)(n + m) + j - k - i}\jap t^{n+m}∂^mχ^\mathup H(\jap tσ^γ)\notag\\
  &= ∑_{\substack{i+k+n+m = ℓ\\i < ℓ,\, k ≤ j}}\jap t^{i+k+n+m}(1 + σ)^μσ^{(γ-1)(i + k + n + m) + j }∂^mχ^\mathup H(\jap tσ^γ)\notag\\
  &\lesssim \jap t^ℓ(1 + σ)^μσ^{j + ℓ(γ-1)}χ^\mathup H(\jap tσ^γ).\notag
  \end{align}
In the low frequency-time region, we likewise have
\begin{align}
  \hspace{2em}&\hspace{-2em}∂^ℓ_σ\big[σ^jφ_t(σ)χ^\mathup L(\jap tσ^γ)\big] = \label{lf-derivator}\\[.4em]
                     &\ANDRA ∑_{\substack{i+k+n+m = ℓ\\i < ℓ,\, k ≤ j}}C_{i,k,n,m}(1 + σ)^μσ^{(γ-1)(n + m) + j - i - k}\jap t^{n+m}∂^mχ^\mathup L(\jap tσ^γ)\notag\\
                     &= ∑_{\substack{i+k+n+m = ℓ\\i < ℓ,\, k ≤ j}}C_{i,k,n,m}(1 + σ)^μ(\jap tσ^γ)^{n + m}σ^{j - i  - k - m - n}∂^mχ^\mathup L(\jap tσ^γ)\notag\\
                     &\lesssim (1 + σ)^μσ^{j-ℓ}χ^\mathup L(\jap tσ^γ)\notag\\
                     &≤ (1 + σ)^{-|I|/2}σ^{j-ℓ}χ^\mathup L(\jap tσ^γ).\notag
\end{align}
We have not yet included the $Φ$ function in the calculations. To do this, notice firstly that since we want to compute $L^2$ norms, since $|λ| ≤ |λ|(d + 2n)$ on the frequency side and since the dependence on $S$ in $Φ(-∆,S)$ is polynomial, it suffices to  treat the case when $Φ(σ,λ) = Ψ(σ) = σ^N/(1 + σ^{2N})$. For this function, we have for all nonnegative $ℓ$ and real numbers $t$ that
\[
  σ^t\,∂^ℓ_σ\Bigg[\frac{(r^{-2}σ)^N}{1 + (r^{-2}σ)^{2N}}\Bigg] ≤ C_{t,ℓ}\, r^{2(t-ℓ)}\frac{(r^{-2}σ)^{N+t-ℓ}}{1 + (r^{-2}σ)^{2N}}.
\]
We shall choose $N$ sufficiently large for the power $N + t - ℓ$ and the difference $2N - (N + t - ℓ) = N + ℓ - t$ to both become as large as need be. As long as these powers are sufficiently large, the expression on the right hand side above – without the leftmost $r$ factor – retains all the properties we need from it, and this relation therefore allows us to trade any finite power of $σ$ for $r^2$, and vice-verse. In effect, we gain a power $σ^{-1} \simeq r^{-2}$ for each derivative acting on $Ψ(r^{-2}σ)$. In the high frequency region, such factors are estimated from above by $\jap tσ^{γ-1}$, which is not larger than the effect of derivatives acting on $φ_t$. In the low frequency region, factors of $σ^{-1}$ provide the same growth that we already have.

The derivative estimates above are therefore in effect also true in the presence of a factor $Ψ(r^{-2}σ)$, with the change that $Ψ$ is replaced with a function with similar properties. For simplicity, we shall use the symbol $Ψ^*$ for functions of the form $Ψ^*(σ) = σ^n/(1 + σ^{2N})$ with sufficiently large $n$ and $N$.

In order to simultaneously study all the terms given to us by Lemma \ref{stora polynomderivatshs}, we let $L$ denote the concatenation of multi-indices $K$ and $J$, and define the vectors
\begin{align*}
  [m/2] &= \big(\ceil{m/2},\ \ceil{m/2} - 1,\ \floor{m/2}\big)\\
  [m/2]' &= \big(\ceil{m/2},\ \ceil{m/2},\ \floor{m/2}\big)\\
  [|L|/2] &= \big(\ceil{|L|/2},\ \floor{|L|/2}, \ceil{|L|/2} - 1\big)\\
  ν &= (0,0,0)\\
  ν' &= (0,0,1),
\end{align*}
which under the requirement that $m + |L|$ is $(\text{even}, \text{odd}, \text{odd})$ satisfy
\[
  -[m/2] - [m/2]' + 2[|L|/2] + ν + ν' =  \big(|L|-m \big)(1,1,1)
\]
We denote the power of $is - \frac14|z|^2$ in the left part of the inner products of the right hand side terms of Lemma \ref{stora polynomderivatshs} by
\[
  ℓ = ℓ(j) = [m/2] - [|L|/2] + j,
\]
and likewise for $ℓ'$ in the right part, and notice that this number is never larger than $[m/2] ≤ [m/2]' ≤ M$. Then, writing $Ψ(r^{-2}\argd)φ_t = φ_t^\mathup H + φ_t^\mathup L$ with $φ_t^\mathup H(σ) = Ψ(r^{-2}σ)φ_t(σ)χ^\mathup H(\jap tσ^γ)$ and likewise for $φ_t^\mathup L$, and using the shorthand $σ = |λ|(d + 2n)$, we can make use of the derivative estimates \pref{hf-derivator} and \pref{lf-derivator} above, and find in the high frequency-time region that for each nonnegative integer $i ≤ ℓ$,
\begin{align*}
  &\FOERSTA\sup_{|ω|≤2m}\big|(nλ)^i∂^{ℓ+i}[\ud σ^{j+ν}φ_t^\mathup H]\big(|λ|(d + 2n + ω)\big)\big| \lesssim\\[-.2em]
  &\ANDRA σ^i \jap t^{ℓ+i}(1 + σ)^μσ^{j+ν+(ℓ+i)(γ-1)}χ^\mathup H(\jap tσ^γ)Ψ^*(r^{-2}σ)\\
  &\lesssim σ^{j+ν+i} \jap t^{ℓ+i}(1 + σ)^μσ^{(ℓ+i)(γ-1)}χ^\mathup H(\jap tσ^γ)Ψ^*(r^{-2}σ)(\jap tσ^γ)^{2M - ℓ - i}\\
  &= \jap t^{2M}σ^{γ2M-ℓ+j+ν}(1 + σ)^μχ^\mathup H(\jap tσ^γ)Ψ^*(r^{-2}σ)\\
  &= \jap t^{2M}σ^{(2γ-1)M - [m/2] + [|L|/2] + ν}(1 + σ)^μχ^\mathup H(\jap tσ^γ)Ψ^*(r^{-2}σ)\\
  &\lesssim \jap t^{2M}r^{2(μ - cM - [m/2] + [|L|/2] + ν)}Ψ^*(r^{-2}σ),
\end{align*}
where we in the last estimate have chosen $c = 1-2γ$. With this estimate at hand, we find, using Lemma \ref{stora polynomderivatshs}, Lemma \ref{nyckelhs}, \hyperlink{biragdial}{the biradial Parseval's formula}, the vector equation above and the fact $|L| = |K| + |J| = |K| + |I|$, that
\begin{align*}
  &\FOERSTA\big\|p_mD^KK^{\mathup HJ}_{\ud r}\big\|_{L^2(ℍ^d×I_k)}^2
  \lesssim \notag\\
  &∫_ℝ|λ|^d∑_{n=0}^∞n^{d-1}\jap t^{4M}2^{2k(2μ + (1-c)2M + |L| - m)}Ψ^*(2^{-2k}|λ|(d + 2n))^2\de λ\notag\\
  &\lesssim \jap t^{4M}2^{2k(Q/2 + 2μ + |L| - c2M + 2M - m)},
\end{align*}
which is in line with estimate \pref{stegtillkrav2}.

For the low frequency-time region, we have instead
\begin{align*}
  &\FOERSTA\sup_{|ω|≤2m}\big|(nλ)^i∂^{ℓ+i}[\ud σ^{j+ν}φ_t^\mathup L]\big(|λ|(d + 2n + ω)\big)\big|\\
  &\lesssim σ^i σ^{-|I|/2+j+ν-(ℓ+i)}χ^\mathup L(\jap tσ^γ)Ψ^*(r^{-2}σ)\\
  &= σ^{-|I|/2+j+ν-ℓ}χ^\mathup L(\jap tσ^γ)Ψ^*(r^{-2}σ)\\
  &= σ^{-|I|/2 + [|L|/2] - [m/2]+ ν }χ^\mathup L(\jap tσ^γ)Ψ^*(r^{-2}σ)\\
  &\lesssim r^{2(- |I|/2 + [|L|/2] - [m/2]+ ν)}Ψ^*(r^{-2}σ).
\end{align*}
This yields, likewise,
\begin{align*}
  \big\|p_mD^KK^{\mathup LJ}_{\ud r}\big\|_{L^2(ℍ^d×I_k)}^2
  &\lesssim ∫_ℝ|λ|^d∑_{n=0}^∞n^{d-1} 2^{2k(-|I|+|L|-m)}Ψ^*(2^{-2k}|λ|(d + 2n))^2\de λ\\
  &\lesssim 2^{2k(Q/2 + |K|-m)},
\end{align*}
which confirms estimate \pref{stegtillkrav2} for this term if we here set $c = 1$. Estimate \pref{stegtillkrav2} is then provided for, and the proof of Theorem \ref{allmaenna s} for $p < 1$ and empty $I'$ is with that completed.

\subsection{Boundedness on $H^1$ and $\BMO$}
Since $D^Iφ_t(-∆)D^{I'}$ is essentially adjoint to $(-1)^{|I|+|I'|}D^{I'}φ_t(-∆)D^I$, the requested estimate for $\BMO$ will follow by duality once the $H^1$ boundedness is known. Just as for $p < 1$, we will however firstly show this in the special case with empty $I'$, with the necessary additions for nonempty $I'$ dealt with in the upcoming subsection.

The method of the previous subsection works also when $p = 1$, with the exception of the proof of the high frequency estimate \pref{krav1}. The problem resides in the last line of estimates \pref{problem foer p = 1}: The decay in $ℓ$ vanishes, and since this matches the decay in $k$, the sum no longer converges for large $k$.

To come around this problem, we shall utilise that while the mentioned kernel method of estimating the $L^2$ norm of $|·|^{2M}Ta$ gives no $ℓ$-dependence, we can when $M = 0$ also estimate the norm directly with Plancherel's theorem (Lemma \ref{L2-multiplikator}), and this will give a dependence on $ℓ$. In order for this property to help us, we shall subdivide the norm of estimate \pref{krav1} along a certain widening cylinder in position-frequency space.

Let $β = γ - 1/4$, so that $γ - 1/2 < β < γ$, and hence both $γ - β$ and $1 - 2(γ - β)$ are positive. Denote by $b_{ℓ,r}$ the characteristic function of the origo-centred ball in $ℍ^d$ of radius
\(
 \jap t 2^{ℓ[1-2(γ-β)]}r^{2β}.
\)
In order to fulfil estimate \pref{krav1}, we shall find an $ε > 0$ such that
\begin{align*}
  &\Bigg\|\Bigg(∫_{2^k}^{2^{k+1}}|T_ra|^2\db r\Bigg)^{1/2}\Bigg\|_{L^1(B^{*c})} ≤\\
  &\Bigg\|\Bigg(∫_{2^k}^{2^{k+1}} b_{ℓ,r}|T_ra|^2\db r\Bigg)^{1/2}\Bigg\|_{L^1(B^{*c})}  + \Bigg\|\Bigg(∫_{2^k}^{2^{k+1}}\big(1 - b_{ℓ,r}\big)|T_ra|^2\db r\Bigg)^{1/2}\Bigg\|_{L^1(B^{*c})}\\
  &\lesssim 2^{-ε(k+ℓ)}t_1, \quad \text{when $k ≥ -\floor{ℓ}$}.
\end{align*}
Write also $B^γ_k$ for the origo-centred ball of radius $\jap t2^{ℓ[1-2(γ-β)]}2^{k2β}$. With the help of Cauchy–Schwarz' inequality, we then see that
\begin{align*}
  &\FOERSTA \Bigg\|\Bigg(∫_{2^k}^{2^{k+1}} b_{ℓ,r}|T_ra|^2\db r\Bigg)^{1/2}\Bigg\|_{L^1(B^{*c})} =\\
  &\Bigg\|\Bigg(∫_{2^k}^{2^{k+1}} b_{ℓ,r}r^{-γQ}\,r^{γQ}|T_ra|^2\db r\Bigg)^{1/2}\Bigg\|_{L^1(ℍ^d)}\\
  &\lesssim \|1\|_{L^2(B^γ_k)}2^{-kγQ}\big\|\ud r^{γQ}ψ(\ud r^{-2}(-∆))Ta\big\|_{L^2(ℍ^d×I_k)}.
\end{align*}
Lemma \ref{utbytarhs} and Lemma \ref{derivatuskshs} then show that this is estimated from above by
\begin{align*}
    &\jap t^{Q/2} 2^{ℓ[1 - 2(γ-β)]Q/2}2^{-k(γ-β)Q}\big\|\ud r^{γQ}D^IΦ(\ud r^{-2}(-∆),\ud r^{-2}S)φ_t(-∆)a\big\|_{L^2(ℍ^d×I_k)}\\[.4em]
  &\lesssim t_1 2^{-(γ-β)Q(k+ℓ)}2^{ℓQ/2}\big\|\ud r^{γQ}(-∆)^{|I|/2}Φ(\ud r^{-2}(-∆), \ud r^{-2}S)φ_t(-∆)a\big\|_{L^2(ℍ^d×I_k)}.
\end{align*}
We can then use the mentioned properties of $Φ$ to replace its operator with the simpler one with $Ψ^*$, thus finding the upper estimate
\[
  t_1 2^{-(γ-β)Q(k+ℓ)}2^{ℓQ/2}\big\|(-∆)^{|I|/2+γQ/2}Ψ^*(\ud r^{-2}(-∆))φ_t(-∆)a\big\|_{L^2(ℍ^d×I_k)},
\]
and Lemma \ref{L2-multiplikator} can then be used to in turn estimate this from above by
\[
  t_1 2^{-(γ-β)Q(k+ℓ)}2^{ℓQ/2}\|a\|_{L^2(ℍ^d×I_k)}.
\]
Finally, the first atomic property bounds this quantity from above by a constant times $t_12^{-(γ-β)Q(k+ℓ)}$, thereby showing that $(γ - β)Q$ is a candidate for $ε$.

For the second part, we shall again make use of estimate \pref{stegtillkrav1}. Here we note that for $p = 1$, it  reads
\[
  \big\|p_mK^J_{\ud r}\big\|_{L^2(ℍ^d×I_k)} ≤ C_M \jap t^{2M}2^{k[(1-2γ)(Q/2 - 2M) + 2M - m]},
\]
for the high frequency-time part of the operator, and for the low frequency-time part, the same estimate holds with $γ$ put to zero. We may therefore treat both parts by assuming only $γ ≥ 0$. We also note that estimates \pref{foersta T^J} and \pref{problem foer p = 1} show that
\[
  \big\||·|^{2M}T_{\ud r}a\|_{L^2(B^{*c}×I_k)} \lesssim 2^{k(m-2M)}∑_{|J|=|I|}∑_{m=0}^{2M}\big\|p_mK_{\ud r}^J\big\|_{L^2(ℍ^d×I_k)}.
\]
We then take a large integer $M$ and use Cauchy–Schwarz' inequality and the estimates above to find that
\begin{align*}
  &\FOERSTA \Bigg\|\Bigg(∫_{2^k}^{2^{k+1}}\big(1 - b_{ℓ,r}\big)|T_ra|^2\db r\Bigg)^{1/2}\Bigg\|_{L^1(B^{*c})} \lesssim\\
     &\ANDRA\big\||·|^{-2M}\big\|_{L^2(B^{γc}_k)}\big\||·|^{2M}T_{\ud r}a\|_{L^2(B^{*c}×I_k)}\\
                &\lesssim \big(\jap t 2^{ℓ[1-2(γ-β)]}2^{k2β}\big)^{Q/2-2M}2^{k(m-2M)}∑_{|J|=|I|}∑_{m=0}^{2M}\big\|p_mK_{\ud r}^J\big\|_{L^2(ℍ^d×I_k)}\\
                  &\lesssim t_1 2^{-[1-2(γ-β)](2M-Q/2)(k+ℓ)},
\end{align*}
so that $[1-2(γ-β)](2M-Q/2)$ is also a candidate for $ε$. Choosing the smaller of the candidates, we provide the requested estimate for the sum of the two terms.

\subsection{Nonempty $I'$}
To deal with the case when $I'$ is not the empty list, we let the smooth $χ^\mathup L : [0,∞)) → [0,1]$ equal $1$ close to the origin, let $χ^\mathup H = 1- χ^\mathup L$, define the function $g$ from $1/g(σ) = χ^\mathup L(σ) + χ^\mathup H(σ) σ^{-|I'|/2}$ and notice that the operator $\jap t^{-Q(1/p-1/2)}D^Iφ_t(-∆)g(-∆)$ is bounded on $H^p$, by the previous result for empty $I'$. Therefore
\begin{align*}
  \|D^Iφ_t(-∆)D^{I'}a\|_{H^p(ℍ^d)} &\lesssim \jap t^{Q(1/p-1/2)}\|χ^\mathup L(-∆)D^{I'}a\|_{H^p(ℍ^d)} \\
                                   &\TERM+\ \jap t^{Q(1/p-1/2)}\|χ^\mathup H(-∆)(-∆)^{-|I'|/2}D^{I'}a\|_{H^p(ℍ^d)}.
\end{align*}
It hence suffices to treat the two norms in the right-hand side above. Since the operators involved are left invariant, the first parts of the proofs for empty $I'$ can still be used, and the difference first appears when one has to deal with $L^2$-norm expressions with position multiplier; expressions of the form
\[
  \big\||·|^{2m}ψ(r^{-2}(-∆))φ(-∆)D^{I'}a\big\|_{L^2(ℍ^d)},
\]
such as in estimates \pref{foersta T_ra-normen}, with trivial multiplier, and in estimates \pref{foersta T^J}. The idea is then to transform the composition of operators in the right hand side above into something that can be dealt with using the previous techniques. More precisely, we shall give an argument that replaces the application of Lemma \ref{utbytarhs} in the proofs for empty $I'$. In the sequel we shall write $I$ instead of $I'$.

For the low frequency term, we firstly notice that identities \pref{kommutera Laplace och D} show that for any positive integer $M$, there is a polynomial $P$ such that
\[
  χ^\mathup L(-∆)D^I = χ^\mathup L(-∆)P(-∆, S)D^I(1 - ∆)^{-M}.
\]
Arguing like in the proof of Lemma \ref{utbytarhs}, one then finds that there are homogeneous polynomials $p_n,q_{2m-n}$ such that, with $A'_K$, $B'$ and $B_K$ as in the mentioned proof,
\begin{align*}
  &\big\||·|^{2m}ψ(r^{-2}(-∆))χ^\mathup L(-∆) D^Ia\|_{L^2(ℍ^d)} ≤\\
  &∑_{|K| = |I|}∑_{n=0}^m\bigg(\big\|r^{n-2m}p_nD^K(1 -∆)^{-M}A'_K(1 + (-r^{-2}∆)^{2N})^{-1}a\big\|_{L^2(ℍ^d)}\\[-.8em]
  &\hspace{12em}×\ \big\|r^{2m-n}q_{2m-n}χ^\mathup L(-∆)P(-∆, S)e^{r^{-2}∆}δ\big\|_{L^1(ℍ^d)} \\[.4em]
  &\hspace{2em} + \big\|r^{n-2m}p_nD^K(1 - ∆)^{-M}B'(1 - (-r^{-2}∆)^{2N})^{-1}a\big\|_{L^2(ℍ^d)}\\[-.2em]
    &\hspace{10em}×\ \big\|r^{2m-n}q_{2m-n}χ^\mathup L(-∆)P(-∆, S)e^{r^{-2}∆}B_Kδ\|_{L^1(ℍ^d)} \bigg).
\end{align*}
Just like identity \pref{multiplikatorutdragning2}, functions of $-∆$ and $S$ satisfy the general dilation property
\[
    φ(r^{-2}(-∆), r^{-2}S)δ = r^Qφ(-∆,S)δ \circ δ_r,
\]
and we can therefore conclude that
\begin{align*}
  &\big\|r^{2m-n} q_{2m-n}χ^\mathup L(-∆)P(-∆, S)e^{r^{-2}∆}B_Kδ\big\|_{L^1(ℍ^d)} \\
  &=\Bigg\|∑_{j=0}^NC'_{K,j}q_{2m-n}χ^\mathup L(-r^2∆)P(-r^2∆, r^2S)e^∆(-∆)^{N-j}S^jδ\Bigg\|_{L^1(ℍ^d)}
\end{align*}
Then we can multiply and divide with $(1 + |·|^2)^M$ for a large $M$, so that Cauchy–Schwarz' inequality shows that this norm is estimated from above by the sum
\[
 ∑_{j=0}^N\big\|(1 + |·|^2)^M(1 + |·|^{2(2m-n)})χ^\mathup L(-r^2∆)P(-r^2∆, r^2S)e^∆(-∆)^{N-j}S^jδ\big\|_{L^2(ℍ^d)},
\]
and we can do similarly for the first term with an $L^1$-norm above. This expression can then be treated with the \hyperlink{biradial}{biradial Plancherel's theorem} and Corollary \ref{utoekad nyckelhs}. Derivatives on the $χ^\mathup L P$-factor will result in a growth in $r^2$, but because of the factor $χ^\mathup L(r^2σ)$ in the Fourier expansion, these are estimated from above by $σ^{-1}$, and we can assume that $N$ is chosen sufficiently large for this to not be a problem. What remains in the low frequency norm is therefore the factor
\begin{align*}
  &\big\|2^{n-2m}p_nD^K(1 - ∆)^{-M}B'(1 - (-r^{-2}∆)^{2N})^{-1}a\big\|_{L^2(ℍ^d)} \\
  &= \big\|2^{n-2m}p_nD^K(1 - ∆)^{-M}Φ'(-r^{-2}∆, r^{-2}S)a\big\|_{L^2(ℍ^d)},
\end{align*}
following the notation of Lemma \ref{utbytarhs}, and similarly for the first term with $A'_K$. If we take $M ≥ |K|/2$, the rest of the proof for empty $I'$ can then be applied to treat this expression, with $μ = γ = t = 0$.

For the high frequency norm, we write firstly 
\[
  χ^\mathup H(-∆)(-∆)^{-|I|/2}D^I = χ^\mathup H(-∆)(-r^{-2}∆)^{-|I|/2}r^{-|I|}D^I.
\]
We then use the technique from the proof of Lemma \ref{utbytarhs} one more time, so that there are homogeneous polynomials $p_n, q_{2m-n}$ such that
\begin{align*}
  &\big\||·|^{2m}ψ(r^{-2}(-∆))χ^\mathup H(-∆)(-∆)^{-|I|/2} D^Ia\|_{L^2(ℍ^d)} ≤\\
  &∑_{|K| = |I|}∑_{n=0}^m\bigg(\big\|r^{n-2m}p_nr^{-|K|}D^KA'_K(1 + (-r^{-2}∆)^{2N})^{-1}a\big\|_{L^2(ℍ^d)}\\[-.8em]
  &\hspace{10em}×\ \big\|r^{2m-n}q_{2m-n}χ^\mathup H(-∆)(-r^{-2}∆)^{-|I|/2}e^{r^{-2}∆}δ\big\|_{L^1(ℍ^d)} \\[.4em]
  &\hspace{2em} + \big\|r^{n-2m}p_nr^{-|K|}D^KB'(1 - (-r^{-2}∆)^{2N})^{-1}a\big\|_{L^2(ℍ^d)}\\[-.4em]
  &\hspace{8em}×\ \big\|r^{2m-n}q_{2m-n}χ^\mathup H(-∆)(-r^{-2}∆)^{-|I|/2}e^{r^{-2}∆}B_Kδ\|_{L^1(ℍ^d)} \bigg).
\end{align*}
The $L^1$-norms can be treated with an argument similar to the low frequency norm. For the $L^2$-norms, the action of Lemma \ref{stora polynomderivatshs} will turn the operator $r^{-|K|}D^K$ into a part of the frequency concentrator operator, again since $N$ can be assumed to be arbitrarily large, and therefore the rest of the proof techniques can be applied to this norm, again with $μ = γ = t = 0$.

\subsection{Boundedness on Lebesgue spaces}
For the cases when $1 < p < ∞$, we invoke an interpolation argument. As pointed out by P.~C.~Kunstmann and M.~Uhl \cite[Fact \capfigures{3.2}]{KunstmannU}, the complex interpolation property of Hardy spaces on Euclidean space shown by S.~Hofmann, S.~Mayboroda and A.~McIntosh \cite[Lemma \capfigures{4.20}]{HofmannMM} can be extended to the setting of homogeneous groups, such as the Heisenberg group. With the common notation, this means that
\[
  [H^{p_0}(ℍ^d),H^{p_1}(ℍ^d)]_θ = H^p(ℍ^d)
\]
in particular in the range when $1/p = (1-θ)/p_0 + θ/p_1$, $1 ≤ p_0, p_1 ≤ 2$ and $0 < θ < 1$. The Lusin norm of the previous section is instrumental in the construction. For a  proof of the Gaussian estimates mentioned in \cite{KunstmannU}, the reader is referred to \cite[Theorem \capfigures{IV.4.2}]{VaropoulosSC}. 

With this property satisfied, we can use complex interpolation of an analytical family of operators on these Hardy spaces. The following is a special case of the extension of L.~Grafakos  \cite[Theorem \capfigures 3]{Grafakos} of the work of Y.~Sagher \cite{Sagher}. See e.\,g.\@\cite{Sagher} for the details on analytical families of operators. 
\begin{theorem}\label{interpoleringss}
  Let $1 ≤ p_0,p_1 ≤ 2$ and let $z \mapsto T_z$ be an analytical family of operators satisfying that for all $f ∈ \Schwartz(ℍ^d)$, there is a $C_f > 0$ and an $a < π$ such that for all $y ∈ ℝ$ and almost all $v ∈ ℍ^d$,
  \[
    \sup_{0 ≤ x ≤ 1} \log |T_{x + iy}f(v)| ≤ C_f e^{a|y|},
  \]
  and suppose that for all $f ∈ \Schwartz(ℍ^d)$, $\|T_{j+iy}f\|_{L^{p_j}(ℍ^d)} ≤ c_j(y)\|f\|_{H^{p_j}(ℍ^d)}$, $j=0,1$, holds for some $c_j(y)$ that satisfy $\log c_j(y) ≤ Be^{b|y|}$ for some $B > 0$, $b < π$. Then for all $x ∈ (0,1)$, there is an $A_x > 0$ such that for all $f ∈ \Schwartz(ℍ^d)$,
  \[
    \|T_xf\|_{L^p(ℍ^d)} ≤ A_x\|f\|_{L^p(ℍ^d)},
  \]
  where we have $1/p = (1-x)/p_0 + x/p_1$.
\end{theorem}
To apply this theorem with $p_0 = 1$, $p_1 = 2$, suppose that $φ_t^*$ satisfies for each $t ≥ 0$ that
\[
  |∂^mφ_t^*(σ)| ≤ C_m\jap t^m(1 + σ)^{-|I|/2-|I'|/2}σ^{m(γ-1)}
\]
and form $T_z = Φ_t^z(-∆)$ with
\[
  Φ_t^z(σ) = φ_t^*(σ)(1+σ)^{-Qγ((1-z)/1 + z/2 -1/2)} = φ_t^*(σ)(1+σ)^{-Qγ(1-z)/2}.
\]
It is clear that all $φ_t$ of Theorem \ref{allmaenna s} with $p ∈ (1,2)$ are of the form $φ_t = Φ_t^x$ with $1/p = (1-x)/1 + x/2$, simply from choosing $φ_t^*(σ) := φ_t(σ)(1 + σ)^{Qγ(1-x)/2}$ and keeping in mind that $γ ≥ 0$. These operators form an analytical family, and we can take a large integer $M$ and use Young's inequality together with Lemma \ref{L2-multiplikator} and Lemma \ref{kaerna aer L2} to show that
\begin{align*}
  |T_zf(v)| &≤ \|(1-∆)^Mf \star T_z(1-∆)^{-M}δ\|_{L^∞(ℍ^d)}\\
  &≤ C_{f,M}\|Φ_t^z(-∆)(1-∆)^{-M}δ\|_{L^2(ℍ^d)}\\
  &≤ C_f\|Φ_t^z\|_{L^∞([0,∞))} ≤ C_f\sup_σ\,(1 + σ)^{-Qγ(1-x)/2} ≤ C_f,
\end{align*}
for $0 ≤ x ≤ 1$. We can find the bounds $c_j(y)$ by following the proof of the boundedness on $H^1$ and $L^2$. Doing so, one finds that the growth in $y$ originates only from derivatives on $Φ_t^z$; we have the polynomial growth of $y^m$ for $∂^mΦ_t^z$. Consequently, since we need only a finite number of derivatives on $Φ_t^z$ for the proof of the $H^1$ boundedness, and none for the $L^2$ one, the requirement $c_j(y) ≤ Be^{b|y|}$ is satisfied.

We can therefore apply the theorem and conclude the requested $L^p$ boundedness of $Φ_t^x$ for $p ∈ (1,2)$. The cases with $p ∈ (2,∞)$ follow from duality.

\section{Applications}\label{tillaempningsavdelning}

\begin{proof}[Proof of Theorem \ref{pd-s}]
With the general sub-Laplacian multipliers from Def\-inition \ref{def multiplikator}, we can construct the operator $e^{it(-∆)^ν}$ and deduce that it solves the Schrödinger system of equations.

From Theorem \ref{maottlig multiplikator-s} and the homomorphism property of sub-Laplacian multipliers, we know that Bessel potentials $(1 - ∆)^{r/2}$ of all real orders $r$ are invertible on $\Schwartz(ℍ^d)$. It hence suffices to show that $φ_t(-∆) = e^{it(-∆)^ν}(1 - ∆)^{-r/2}$ with $r = 2νQ|1/p - 1/2|$ is bounded on $H^p(ℍ^d)$ and $\BMO(ℍ^d)$.

The operator $φ_t(-∆)$ can be separated into low and high frequency-time parts, each satisfying the requirements of Theorem \ref{allmaenna s}. To be precise, let $χ^\mathup L : \Smooth_c([0,∞))$ equal $1$ close to $0$ and set $χ^\mathup H = 1 - χ^\mathup L$. Then $φ_t = φ_t^\mathup L + φ_t^\mathup H$ with $μ = -r/2$, $\jap t = 1 + t$ and
\begin{align*}
  φ_t^\mathup L(σ) &= e^{itσ^ν}(1 + σ)^μχ^\mathup L(\jap tσ^ν)\\
  φ_t^\mathup H(σ) &= e^{itσ^ν}(1 + σ)^μχ^\mathup H(\jap tσ^ν).
\end{align*}
The necessary derivative estimates of such functions were computed in detail in estimates \pref{hf-derivator} and \pref{lf-derivator}, if one sets $j = 0$, and from these we find that $φ_t^\mathup H$ satisfies the condition of Theorem \ref{allmaenna s} with $γ = ν$ and $|I| = |I'| = 0$ with the choice
\[
    μ = -νQ|1/p - 1/2| = -r/2.
\]
And even with this power fixed, $φ_t^\mathup L$ satisfies the condition of Theorem \ref{allmaenna s} with $γ = 0$, $μ = 0$ and $|I| = |I'| = 0$. The theorem can hence be applied also to the low frequency-time operator, and Theorem \ref{pd-s} follows.
\end{proof}

\begin{proof}[Proof of Theorem \ref{Besselpotentials}]
  Per the comment after Theorem \ref{allmaenna s}, we need only show the leftmost estimate
  \[
    \|(1 - ∆)^{n/2}f\|_{H^p(ℍ^d)} \lesssim ∑_{|I| ≤ n}\|D^If\|_{H^p(ℍ^d)}
  \]
  for integer $n$, and likewise for $\BMO$. This is trivial for even $n$. For the cases when $n$ is odd, we use that
  \[
    (1 - ∆)^{n/2} = ∑_{i=1}^{2d}(1 - ∆)^{-1/2}\Big(\textstyle\frac1{\sqrt{2d}} + E_i\Big)\Big(\textstyle\frac1{\sqrt{2d}} - E_i\Big)(1 - ∆)^{(n-1)/2},
  \]
  and therefore, since the boundedness of the operators $(1 - ∆)^{-1/2}(1/\sqrt{2d} + E_i)$ is shown by Theorem \ref{allmaenna s}, the requested estimate follows.
\end{proof}

\subsubsection*{Acknowledgements} I am greatly indebted to Isabelle Gallagher for many enlightening and encouraging discussions, which have been instrumental in my progress and understanding.

I am also deeply grateful to the Wenner-Gren Foundations, which have funded this research. The larger part of this project was carried out while I was working at the École normale supérieure in Paris.

\printbibliography

{\itshape\noindent
  Aksel Bergfeldt\\
  Institut de Mathématiques de Jussieu-Paris Rive Gauche\\
  Université Paris Cité, Bâtiment Sophie Germain, Boite Courrier 7012\\
  8 Place Aurélie Nemours\\
  75205 Paris – Cedex 13}\\[.4em]
E-mail: \textit{aksel.bergfeldt@imj-prg.fr}

\end{document}